\documentclass[review,10pt]{elsarticle}
\usepackage{lipsum}
\usepackage{amsthm}
\makeatletter
\def\ps@pprintTitle{%
 \let\@oddhead\@empty
 \let\@evenhead\@empty
 \def\@oddfoot{}%
 \let\@evenfoot\@oddfoot}
\makeatother
\usepackage[margin=2.5cm]{geometry}
\usepackage{lineno,hyperref}
\usepackage{amsmath,amssymb,latexsym, amscd}
\usepackage{exscale, eps fig, graphics}
\usepackage{array}
\usepackage[usenames,dvipsnames]{color}
\usepackage{float}
\restylefloat{table}
\newcommand\numberthis{\addtocounter{equation}{1}\tag{\theequation}}
\usepackage{algorithmic}
\usepackage{placeins}
\usepackage{longtable}
\usepackage{makecell}
\usepackage[utf8]{inputenc}
\usepackage[justification=centering]{caption}
\usepackage{graphicx}
\usepackage{caption}
\usepackage{subcaption}
\usepackage{mathtools}
\usepackage{eucal}
\usepackage{mathrsfs}
\usepackage{fontenc}
\hypersetup{pdfauthor={Name}}
\newtheorem{thm}{Theorem}[section]

\newtheorem{prop}[thm]{Proposition}
\newtheorem{rem}[thm]{Remark}
\newtheorem{conjecture}[thm]{Conjecture}

\newtheorem{cor}[thm]{Corollary}

\newcommand{\BigO}[1]{\ensuremath{\operatorname{O}\left(#1\right)}}
\newcommand{\BigOk}[1]{\ensuremath{\operatorname{O}_k\left(#1\right)}}

\renewcommand{\geq}{\geqslant}
\renewcommand{\leq}{\leqslant}

\newcommand{\BigOkq}[1]{\ensuremath{\operatorname{O}_{q, k}\left(#1\right)}}
\newcommand{\BigOqkd}[1]{\ensuremath{\operatorname{O}_{\mathfrak{q},k,d}\left(#1\right)}}
\newcommand{\BigOqk}[1]{\ensuremath{\operatorname{O}_{\mathfrak{q},k}\left(#1\right)}}
\newcommand{\BigOmk}[1]{\ensuremath{\operatorname{O}_{\mathfrak{m},k}\left(#1\right)}}

\begin{document}

\begin{frontmatter}

\title{Chebyshev's bias for irrational factor function }

\author[mymainaddress]{Bittu Chahal}
\ead{bittui@iiitd.ac.in}

%\author[mymainaddress]{Sneha Chaubey}
%\ead{sneha@iiitd.ac.in}

\address[mymainaddress]{Department of Mathematics, IIIT Delhi, New Delhi 110020}

\begin{abstract}
In this article, we study the distribution of the irrational factor function of order $k$, introduced first by Atanassov for $k=2$ and later it was generalized by Dong et al. for all $k\geq 2$. We introduce the irrational factor function in both number field and function field settings, derive asymptotic formulas for their average value, and further establish omega results for the error term in the asymptotic formulas. Moreover, we study the Chebyshev's bias phenomenon for number field and function field analogues of sum of the irrational factor function. 

%In particular, we establish the asymptotic formula for the average of the irrational factor function and provide an $\Omega$-result for its error term. Moreover, we study the Chebyshev's bias phenomenon for the sum of irrational factor function of order $k$ along arithmetic progressions. Furthermore, we introduce the irrational factor function in both number field and function field settings and derive asymptotic of its average value. Finally, we study the Chebyshev's bias phenomenon for number field and function field analogues of sum of the irrational factor function. 

\end{abstract}


\end{frontmatter}

\section{Introduction and main results}
In 1837, Dirichlet \cite{Dirichlet} proved that for any integers $a$ and $q$ with $(a,q)=1$, there are infinitely many primes $p\equiv a\pmod{q}$, and are evenly distributed along arithmetic progressions. In a letter to M. Fuss in 1853, Chebyshev \cite{Chebyshev} remarked that the prime quadratic non-residues of a given modulus exceeds that of the prime quadratic residues. This is called the Chebyshev's bias phenomenon. Since Chebyshev's note, several papers have been written on this topic in an attempt to understand in what sense his assertion is true. Littlewood \cite{Littlewood} in 1914 proved that the quantities $\pi(x; 3,4)-\pi(x;1,4)$ and $\pi(x; 2,3)-\pi(x;1,3)$ change sign infinitely often. Furthermore, the study of prime number races, which involves examining inequalities among the counting functions of prime numbers in arithmetic progression, was developed by Shanks \cite{Shanks}, Knapowski and Tur\'{a}n in a series of papers \cite{Knapowski, Knapowski2, Knapowski3, Knapowski4, Knapowski5, Knapowski6, Knapowski7, Knapowski8}, as well as by Kaczorowski \cite{Kaczorowski1, Kaczorowski2, Kaczorowski3}. Rubinstein and Sarnak \cite{Sarnak} made one of the most significant contributions by proving that the set of positive real numbers $x$ for which the inequalities $\pi(x;q,a_1)>\cdots>\pi(x;q,a_r)$ hold has a positive density, under certain hypotheses on the zeros of the Dirichlet $L$-functions modulo $q$. For more details and comprehensive surveys, one is referred to the survey articles of Granville and Martin \cite{Granville} and of Martin et al. \cite{Martin}.
%Rubinstein and Sarnak \cite{Sarnak} made one of the most significant contributions by providing the justification for the Chebyshev's bias under certain plausible hypothesis. In particular, they showed that the set of positive real numbers $x$ for which the inequality holds $\pi(x;q,a_1)>\cdots>\pi(x;q,a_r)$ has a positive density, under the hypothesis on the zeros of Dirichlet $L$-function modulo $q$.

 Littlewood's result can be interpreted in other words, i.e., the set of values of $x$ for which the difference 
 \[\pi(x; 3,4)-\pi(x;1,4) \numberthis\label{diff}\]
 is positive and the set of values of $x$ for which the difference in \eqref{diff} is negative are unbounded. The difference in \eqref{diff} can be written as
 \[\sum_{\substack{n\leq x\\n\equiv 3\pmod{4}}}\frac{\Lambda(n)\mu(n)^2}{\log n}-\sum_{\substack{n\leq x\\n\equiv 1\pmod{4}}}\frac{\Lambda(n)\mu(n)^2}{\log n}, \numberthis\label{lambda}\]
where $\Lambda(n)$ and $\mu(n)$ are the von Mangoldt and M\"{o}bius functions, respectively. In this paper, we study the analogue of Littlewood's result in the form of \eqref{lambda} for arithmetic function $I_k(n)$, the irrational factor function of order $k$.

\subsection{Irrational factor function over $\mathbb{Q}$}
Let $n$ be a positive integer, and let $n=\prod_{i=1}^lp_i^{\alpha_i}$ be its prime factorization. Atanassov \cite{Atanassov} introduced the irrational factor function of $n$ as follows:
\[ I(n)=\prod_{i=1}^lp^{1/\alpha_i}.\]
%Alkan et al. \cite{Alkan} studied the asymptotic behavior of the irrational factor function $I(n)$. 
Let $S(n)$ denote the square-free part of $n$. Alkan et al. \cite{Alkan} observed that $I(n)\geq S(n)^{1/(k-1)}\geq n^{1/(k-1)^2}$ and $I(n)\leq S(n)^{1/k}$ when $n$ is $k$-power free and $k$-power full, respectively. Consequently, $I(n)$ measures how far $n$ is from being $k$-power free or $k$-power full. In particular, small values of $I(n)$ are associated to $k$-power full integers, while large values of $I(n)$ correspond to $k$-power free integers. Nevertheless, $I(n)$ is a rough approximation, since $k$ does not appear in its definition.
Motivated by the observation of Alkan et al. \cite{Alkan}, Dong et al. \cite{Meng} defined the irrational factor function of order $k$ by introducing the parameter $k$ into the function $I(n)$. For any integer $k\geq 2$, it is defined as:
\[I_k(n)=\prod_{i=1}^lp_i^{\beta_i}, \] 
where \[
\beta_i=\left\{\begin{array}{cc}
   \alpha_i,  & \mbox{if} \ \alpha_i<k,\\
  \frac{1}{\alpha_i},   & \mbox{if}\ \alpha_i\geq k.
\end{array}\right. \numberthis\label{beta} \]
It is clear that $I_2(n)=I(n)$. It was noticed by Dong et al. \cite{Meng} that $\frac{I_k(n)}{n}$ closely approximates $\mu_k(n)$, where $\mu_k(n)$ denotes the characteristic function of $k$-free integers. In particular,
\[\left|\frac{I_k(n)}{n}-\mu_k(n)\right|\leq \begin{cases}
   0,  & \mbox{if} \ n\ \mbox{is}\ k\mbox{-free},\\
  \frac{1}{n^{1-\frac{1}{k^2}}},   & \mbox{if}\ n\ \mbox{is}\ k\mbox{-power full}.
\end{cases} \]
Moreover, they also established the asymptotic formula for the average of $I_k(n)$ and proved that the average of $\frac{I_k(n)}{n}$ is $\frac{c_k}{\zeta(k)}$, where $c_k$ is a constant. It is well known that (see \cite{Walfisz})
\[\sum_{n\leq x}\mu_k(n)=\frac{x}{\zeta(k)}+\BigO{x^{\frac{1}{k}}\exp{(-ck^{-\frac{8}{5}}(\log x)^{\frac{3}{5}}(\log\log x)^{-\frac{1}{5}})}}. \]
 Therefore, on average, $\frac{I_k(n)}{n}$ and $\mu_k(n)$ are very close for large $k$. Motivated by this, we investigate the distribution of the irrational factor function $I_k(n)$. 
%We begin by deriving the asymptotic formula for the average of $I_k(n)$ along arithmetic progressions.
Let $a$ and $q$ be positive integers. 
 We denote
\[S_k(Q;q,a):=\sum_{\substack{n\leq Q\\ n\equiv a\pmod{q}}}I_k(n) . \numberthis\label{S_k}\]
 Alkan et al. \cite{Alkan} and Dong et al. \cite{Meng} studied the asymptotic behavior of $S_2(Q;1,a)$ and $S_k(Q;1,a)$, respectively. The main objective of this article is to introduce the irrational factor function over number fields and function fields, study asymptotic behavior for the average of its analogs over arithmetic progressions, and investigate the Chebyshev's bias phenomenon. By setting the number field to the field of rationals, $\mathbb{Q}$, we obtain the results for the distribution of the irrational factor function over $\mathbb{Q}$ as a special case.

\subsection{Irrational factor function over number fields}
%The main objective of this section is to introduce the irrational factor function over number fields and establish the asymptotic formula for the average of the analogous irrational factor function over arithmetic progressions and investigate the Chebyshev's bias phenomenon.
Let $\mathbb{K}$ be a number field of degree $m$ and $\mathcal{O}_{\mathbb{K}}$ be its ring of integers. Let $\mathfrak{I}\subset\mathcal{O}_{\mathbb{K}}$ be a non-zero ideal, and let $\mathfrak{I}=\prod_{i=1}^{l}\mathfrak{p}_i^{\alpha_i}$ be its factorization into non-zero prime ideals $\mathfrak{p}_i$. %Every non-zero ideal $\mathfrak{I}\subset\mathcal{O}_{\mathbb{K}}$ can be written as follows: $\mathfrak{I}=\prod_{i=1}^{l}\mathfrak{p}_i^{\alpha_i},$ where $\mathfrak{p}_i$ are prime ideals. 
We define the irrational factor function for $\mathfrak{I}\subset\mathcal{O}_{\mathbb{K}}$ as follows:
\[\mathcal{F}_{k}(\mathfrak{I})=\prod_{i=1}^{l}(\mathcal{N}\mathfrak{p}_i)^{\beta_i}, \]
where $\beta_i,\ 1\leq i\leq l$ are as in \eqref{beta} and $\mathcal{N}\mathfrak{p}$ denotes the norm of an ideal $\mathfrak{p}$. In order to study the behavior of $\mathcal{F}_{k}(\mathfrak{I})$, we first establish the asymptotic formula for the average of the $\mathcal{F}_{k}(\mathfrak{I})$. Denote
\[\mathcal{S}_{k,\mathbb{K}}(x):=\sum_{\substack{\mathfrak{I}\subset\mathcal{O}_{\mathbb{K}}\\ \mathcal{N}(\mathfrak{I})\leq x}}\mathcal{F}_{k}(\mathfrak{I}). \]

 To state our results, let us fix some notation. Let $\zeta_{\mathbb{K}}(s)$ be the Dedekind zeta function associated with the number field $\mathbb{K}$, and let $\mathcal{L}(s,\chi)$ denote the Hecke $L$-function associated with Hecke character $\chi\pmod{\mathfrak{q}}$ of $\mathcal{O}_{\mathbb{K}}$. The key idea to obtain the asymptotic formula for $\mathcal{S}_{k,\mathbb{K}}(x)$ is the Wiener-Ikehara Tauberian theorem. Our first result is as follows:
\begin{thm}\label{thm6 Tauberian}
 For $k\geq 2$, we have
  \[\mathcal{S}_{k,\mathbb{K}}(x)=\frac{x^2\lambda_{\mathbb{K}}{R}_{k}(2)}{2\zeta_{\mathbb{K}}(k)}+\BigOk{\frac{x^2}{\log x}}, \] 
  where $\lambda_{\mathbb{K}}$ is the residue of the Dedekind zeta function at $s=1$ and ${R}_{k}(2)$ is a constant depending on $k$.
\end{thm}
For $\mathbb{K}=\mathbb{Q}$, we immediately obtain the result of Dong et al. \cite[Theorem 2]{Meng}. We next study the asymptotic behavior of the irrational factor function $\mathcal{F}_{k}(\mathfrak{I})$ over arithmetic progressions.
Let $\mathfrak{a}, \mathfrak{q}\subset\mathcal{O}_{\mathbb{K}}$ be fixed integral ideals. Denote
\[\mathcal{S}_{k,\mathbb{K}}(x;\mathfrak{q},\mathfrak{a}):=\sum_{\substack{\mathfrak{I}\subset\mathcal{O}_{\mathbb{K}}\\ \mathcal{N}(\mathfrak{I})\leq x\\ \mathfrak{I}\equiv \mathfrak{a}\pmod{\mathfrak{q}}}}\mathcal{F}_{k}(\mathfrak{I}). \numberthis\label{Ns1}\]
In order to derive our next result, we first state the Generalized Lindel\"{o}f hypothesis (GLH).
\begin{conjecture}(Generalized Lindel\"{o}f hypothesis).
 Let $\mathfrak{q}\subset\mathcal{O}_{\mathbb{K}}$ be a fixed integral ideal and let $\mathcal{L}(s,\chi)$ be the Hecke $L$-function modulo $\mathfrak{q}$. Then for any $\epsilon>0$, we have
   \[\mathcal{L}\left(\frac{1}{2}+it,\chi\right)\ll_{\mathfrak{q}}|t|^{\epsilon}. \]
\end{conjecture}

Following this, we present our next result.
\begin{thm}\label{thm4}
  Let $\mathbb{K}$ be a number field of degree $m$. Let $\mathfrak{a}, \mathfrak{q}\subset\mathcal{O}_{\mathbb{K}}$ be fixed non-zero integral ideals with $(\mathfrak{a}, \mathfrak{q})=1$. Then, for $k\geq 2$, we have
  \[\mathcal{S}_{k,\mathbb{K}}(x;\mathfrak{q},\mathfrak{a})=\frac{x^2\lambda_{\mathbb{K}}{R}_{k}(2,\chi_0)}{2|(\mathcal{O}_\mathbb{K}/\mathfrak{q})^*|\mathcal{L}(k,\chi_0)}\prod_{\mathfrak{p}|\mathfrak{q}}\left(1-\frac{1}{\mathcal{N}(\mathfrak{p})}\right)+\mathcal{E}(x), \]
  where
  \[
\mathcal{E}(x)=\begin{cases}
 \BigOkq{x^{\frac{2(2k-1)}{3k-2}}\exp{\left(-c\left(\frac{\log x}{\log\log x}\right)^{\frac{1}{3}}\right) }},  &\mbox{unconditionally and}\ m=1,\\
   \BigOqkd{x^{1+\frac{1}{k}+\frac{m(k-1)^2}{2k^2}+\epsilon}},  & \mbox{unconditionally and}\ m=2,3  , \\
  \BigOqk{x^{\frac{k+1}{k}+\epsilon}},   & \mbox{on the GLH and}\ m\geq 4,
\end{cases} \]
 $\lambda_{\mathbb{K}}$ is the residue of the Dedekind zeta function at $s=1$ and ${R}_{k}(2,\chi_0)$ is a constant depending on $k$ and $\mathfrak{q}$.
\end{thm}

For $m\geq 4$, the above theorem provides the asymptotic formula for $\mathcal{S}_{k,\mathbb{K}}(x;\mathfrak{q},\mathfrak{a})$, assuming the Generalized Lindel\"{o}f hypothesis. In order to obtain unconditional result, we invoke the Wiener-Ikehara Tauberian theorem and derive the bound for $\mathcal{S}_{k,\mathbb{K}}(x;\mathfrak{q},\mathfrak{a})$. We then have the following result.

\begin{thm}\label{thm6 Tauberian2}
 Let $\mathfrak{a}, \mathfrak{q}\subset\mathcal{O}_{\mathbb{K}}$ be fixed non-zero integral ideals with $(\mathfrak{a}, \mathfrak{q})=1$. For $k\geq 2$, we have
   \[\mathcal{S}_{k,\mathbb{K}}(x;\mathfrak{q},\mathfrak{a})\ll_{\mathfrak{q}} \frac{x^2\lambda_{\mathbb{K}}{R}_{k}(2)}{\zeta_{\mathbb{K}}(k)}, \]
    where $\lambda_{\mathbb{K}}$ is the residue of the Dedekind zeta function at $s=1$ and ${R}_{k}(2)$ is a constant depending on $k$.
\end{thm}

In particular, for the classical case $\mathbb{K}=\mathbb{Q}$, we immediately obtain the asymptotic formula for $S_k(Q;q,a)$ as a consequence of Theorem \ref{thm4}.

\begin{cor}
   Suppose that $a$ and $q$ are positive integers with $\gcd(a,q)=1$. Then, for $k\geq 2$, we have
    \[S_k(Q;q,a)=c(q,a)Q^2+\BigOkq{Q^{\frac{2(2k-1)}{3k-2}}\exp{\left(-c\left(\frac{\log Q}{\log\log Q}\right)^{\frac{1}{3}}\right) }}, \]
    where $c(q,a)=\frac{M_{k,\chi_0}(2)}{2\phi(q)L(k,\chi_0)}\prod_{p|q}\left(1-\frac{1}{p} \right)$ and $M_{k,\chi_0}(2)$ is a constant depending on $k$ and $q$.  
\end{cor}

The above corollary provides the value of $c(q,a)$, when $(a,q)= 1$. In order to obtain the value of $c(5,0)$, we combine the above corollary with the result of Dong et al. \cite[Theorem 2]{Meng} concerning the asymptotic formula for the average of $I_k(n)$, over all $n\leq x$. Therefore,
\[c(5,0)=\frac{K_k(2)}{2\zeta(3)}-\frac{2M_{k,\chi_0}(2)}{\phi(q)L(k,\chi_0)}\prod_{p|q}\left(1-\frac{1}{p} \right), \]
where $K_k(s)$ is some specific function (see \cite[p. 358]{Meng}). It is evident that $c(5,a)=c(5,1)$ for $a\in\{2,3,4\}$, and numerical computation shows that $c(5,0)<c(5,1)$. This, in turn, implies that $S_k(Q;5,0)<S_k(Q;5,1)$ for sufficiently large $Q$. Thus, the sum $S_k(Q;q,0)$ is strongly biased toward the sum $S_k(Q;5,a)$ when $(a,5)=1$. The above corollary does not provide any insight into the bias between the summatory functions $S_k(Q;q,a_1)$ and $ S_k(Q;q,a_2)$ when $(a_1a_2,q)=1$. Therefore, to compare their growth, we prove a result analogous to the Littlewood's theorem.

Further, we prove the $\Omega$-result for the error term in the asymptotic formula of Theorem \ref{thm4}. We first state the Generalized Haselgrove's condition for the zeros of the Hecke $L$-function modulo $\mathfrak{q}$.

\noindent
\textbf{Haselgrove's condition for Hecke $L$-function modulo $\mathfrak{q}$}: For all Hecke characters $\chi$ (mod $\mathfrak{q}$), $\mathcal{L}(s,\chi)\ne 0$ for all $s\in(0,1)$.

Note that J. B. Rosser \cite{Rosser1, Rosser2} showed that no Dirichlet $L$-function attached to a real character modulo $q\leq 1000$ has a real zero in the strip $0<\Re(s)<1$, and Watkins \cite{Watkins} proved that one can take $q\leq 300000000$ if we restrict the Dirichlet $L$-functions to odd characters. For $q|24$, there are only real characters modulo $q$. Therefore, in the classical case $\mathbb{K}=\mathbb{Q}$, Haselgrove's condition is known to hold when modulus $q$ divides $24$.
 We proceed for our next result.
\begin{thm}\label{thm5}
 Let $\mathfrak{a}, \mathfrak{q}\subset\mathcal{O}_{\mathbb{K}}$ be fixed non-zero integral ideals with $(\mathfrak{a}, \mathfrak{q})=1$. Let $\Theta$ denote the supremum of the real part of zeros of the Hecke $L$-function modulo $\mathfrak{q}$. Assuming Haselgrove's condition for the Hecke $L$-function modulo $\mathfrak{q}$, we have
  \[\mathcal{S}_{k,\mathbb{K}}(x;\mathfrak{q},\mathfrak{a})-\frac{x^2\lambda_{\mathbb{K}}{R}_{k}(2,\chi_0)}{2|(\mathcal{O}_\mathbb{K}/\mathfrak{q})^*|\mathcal{L}(k,\chi_0)}\prod_{\mathfrak{p}|\mathfrak{q}}\left(1-\frac{1}{\mathcal{N}(\mathfrak{p})}\right)=\Omega_{\pm}\left(x^{1+\frac{\Theta}{k}-\epsilon}\right). \]  
\end{thm}

Next, we prove the following result for the sign changes of $\mathcal{S}_{k,\mathbb{K}}(x;\mathfrak{q},\mathfrak{a}_1)-\mathcal{S}_{k,\mathbb{K}}(x;\mathfrak{q},\mathfrak{a}_2)$ analogous to Littlewood's theorem.
\begin{thm}\label{thm6}
 Let $\mathfrak{q}\subset\mathcal{O}_{\mathbb{K}}$ be a fixed non-zero integral ideal and $\mathfrak{a}_1, \mathfrak{a}_2\subset\mathcal{O}_{\mathbb{K}}$ be ideals such that $\mathfrak{a}_1\not\equiv\mathfrak{a}_2\pmod{\mathfrak{q}}$ and $(\mathfrak{a}_1\mathfrak{a}_2,\mathfrak{q})=1$. Assuming Haselgrove's condition for Hecke $L$-function modulo $\mathfrak{q}$, the set of values of $x$ for which the difference $\mathcal{S}_{k,\mathbb{K}}(x;\mathfrak{q},\mathfrak{a}_1)-\mathcal{S}_{k,\mathbb{K}}(x;\mathfrak{q},\mathfrak{a}_2)$ is strictly positive and the set of values of $x$ for which the difference $\mathcal{S}_{k,\mathbb{K}}(x;\mathfrak{q},\mathfrak{a}_1)-\mathcal{S}_{k,\mathbb{K}}(x;\mathfrak{q},\mathfrak{a}_2)$ is strictly negative are unbounded.  
\end{thm}

%In order to investigate the distribution of irreducible monic polynomials in a polynomial ring over a finite field, the analog of Chebyshev’s bias phenomenon has been studied in the function field setting (see \cite{Cha}).
The aim of the next section is to introduce the irrational factor function over function fields and investigate its distribution.
\subsection{Irrational factor function over function fields}
Let $\mathbb{F}_q$ be a finite field with $q$ elements and let $A=\mathbb{F}_q[X]$ be a polynomial ring over $\mathbb{F}_q$. We denote the degree of a polynomial $f\in\mathbb{F}_q[X]$ by $\deg(f)$, and its norm $|f|$ is defined as $q^{\deg(f)}$. Let $\Phi(f)$ denote the number of elements in the group $(A/fA)^*$. Every non-zero polynomial $f\in\mathbb{F}_q[X]$ can be written in the form
$f(X)=\alpha P_1^{\alpha_1}P_2^{\alpha_2}\cdots P_l^{\alpha_l}$,
where $P_i$ are monic irreducible polynomials. The irrational factor function for $f\in\mathbb{F}_q[X]$ is defined as:
\[\mathfrak{F}_{k}(f)=\prod_{i=1}^{l}|P_i|^{\beta_i}, \]
where $\beta_i,\ 1\leq i\leq l$ are as in \eqref{beta}.
Let $ \mathfrak{g}, \mathfrak{m}\in\mathbb{F}_q[X]$ be non-zero polynomials. We denote
\[\mathfrak{S}_{k}(N; \mathfrak{m},\mathfrak{g})=\sum_{\substack{f\in\mathbb{F}_q[X]\\\deg(f)\leq N\\f\equiv\mathfrak{g}\pmod{\mathfrak{m}}}}\mathfrak{F}_{k}(f). \]

To describe our results, we fix some notation. Let $\chi$ be a Dirichlet character modulo $\mathfrak{m}$ over the polynomial ring $A$, and let $L_q(s,\chi)$ denote the corresponding Dirichlet $L$-function modulo $\mathfrak{m}$.
In this section, our first result concerns the asymptotic formula for $\mathfrak{S}_{k}(N; \mathfrak{m},\mathfrak{g})$.
%We first obtain an asymptotic formula for the average of irrational factor function over function field.
\begin{thm}\label{thm7}
 Let $q\geq 2$ be a fixed integer. Let $\mathfrak{m}, \mathfrak{g}\in\mathbb{F}_q[X]$ be fixed non-zero polynomials such that $(\mathfrak{g},\mathfrak{m})=1$. For $k\geq 2$, we have
    \[\mathfrak{S}_{k}(N; \mathfrak{m},\mathfrak{g})=\frac{q^{2N}\mathcal{M}_k(2,\chi_0)}{2 \Phi(\mathfrak{m})L_q(k,\chi_0)\log q}\prod_{P|\mathfrak{m}}\left(1-\frac{1}{|P|} \right)+\BigOmk{q^{N(1+\frac{1}{2k}+\epsilon)}}, \]
    where $\mathcal{M}_k(2,\chi_0)$ is a constant depending on $k$ and $\mathfrak{m}$.
\end{thm}

Next, we derive an $\Omega$-result for the error term in the above theorem. It is well known that the Riemann hypothesis is true in function fields, thus in terms of omega result, we obtain only an $\epsilon$ saving in the error term of the above theorem. The Haselgrove's condition on the zeros of the Dirichlet $L$-function modulo $\mathfrak{m}$ over function fields becomes $L_q(1/2,\chi)\ne 0$. Thus, our next result is the following theorem.
\begin{thm}\label{thm9}
   Let $q\geq 2$ be a fixed integer. Let $\mathfrak{m}, \mathfrak{g}\in\mathbb{F}_q[X]$ be fixed non-zero polynomials with $(\mathfrak{g},\mathfrak{m})=1$. Assume $L_q(1/2,\chi)\ne 0$. Then, for $k\geq 2$, we have
   \[\mathfrak{S}_{k}(N; \mathfrak{m},\mathfrak{g})-\frac{q^{2N}\mathcal{M}_k(2,\chi_0)}{2\Phi(\mathfrak{m})L_q(k,\chi_0)\log q}\prod_{P|\mathfrak{m}}\left(1-\frac{1}{|P|} \right)=\Omega_{\pm}(q^{N(1+\frac{1}{2k}-\epsilon)}). \]
\end{thm}

We now proceed to our next result concerning the sign changes of $\mathfrak{S}_{k}(N; \mathfrak{m},\mathfrak{g})$.
\begin{thm}\label{thm8}
   Let $q\geq 2$ be an integer. Let $\mathfrak{m}\in\mathbb{F}_q[X]$ be a fixed non-zero polynomial and $\mathfrak{g}_1, \mathfrak{g}_2\in\mathbb{F}_q[X]$ such that $\mathfrak{g}_1\not\equiv\mathfrak{g}_2\pmod{\mathfrak{m}}$ with  $(\mathfrak{g}_1\mathfrak{g}_2,\mathfrak{m})=1$. Suppose $k\geq 2$ is an integer such that $\chi^k\ne\chi_0$ for any non-principal character $\chi\pmod{\mathfrak{m}}$. Assume that $L_q(1/2,\chi)\ne 0$. Then the function $\mathfrak{S}_{k}(N; \mathfrak{m},\mathfrak{g}_1)-\mathfrak{S}_{k}(N; \mathfrak{m},\mathfrak{g}_2)$ changes sign infinitely often.
\end{thm}

\begin{rem}
    Note that Theorems \ref{thm6}, and \ref{thm8} do not provide any information on the frequency of sign changes of the functions $\mathcal{S}_{k,\mathbb{K}}(x;\mathfrak{q},\mathfrak{a}_1)-\mathcal{S}_{k,\mathbb{K}}(x;\mathfrak{q},\mathfrak{a}_2)$ 
    and $\mathfrak{S}_{k}(N; \mathfrak{m},\mathfrak{g}_1)-\mathfrak{S}_{k}(N; \mathfrak{m},\mathfrak{g}_2)$, respectively. We use a result of Kaczorowski and Wiertelak \cite[Lemma 3.1]{KaczorowskiJ} to get some insights into the frequency of sign changes. Employing \cite[Lemma 3.1]{KaczorowskiJ} to the function
    \[\mathcal{A}(x)=\mathcal{S}_{k,\mathbb{K}}(x;\mathfrak{q},\mathfrak{a}_1)-\mathcal{S}_{k,\mathbb{K}}(x;\mathfrak{q},\mathfrak{a}_2)\pm cx^{1+\frac{1}{2k}-\epsilon}, \]
    we get a sequence $\{x_i\}_{i=1}^{[\log T]}$ in the interval $(1,T]$ of length $\log T$ such that $sgn \mathcal{A}(x_i)\ne sgn \mathcal{A}(x_{i+1})$ and $|\mathcal{A}(x_i)|>x_i^{1+\frac{1}{2k}-\epsilon}$. Hence, $\mathcal{A}(x)$ has at least $\gg\log T$ oscillations of size $x^{1+\frac{1}{2k}-\epsilon}$, in the interval $(1,T]$.
\end{rem}

Similarly, $\mathbb{A}(N)=\mathfrak{S}_{k}(N; \mathfrak{m},\mathfrak{g}_1)-\mathfrak{S}_{k}(N; \mathfrak{m},\mathfrak{g}_2)\pm cq^{N(1+\frac{1}{2k}-\epsilon)}$ demonstrate at least $\gg\log T$ oscillations of size $q^{N(1+\frac{1}{2k}-\epsilon)}$, in the interval $(1,T]$.

%\subsection{Remarks on proofs}
%In order to prove the asymptotic formula, the Perron's formula and its analog in number fields and function fields with the bounds of the Riemann zeta function, Dirichlet $L$-function, Dedekind zeta function, Hecke $L$-function, zeta function and $L$-function for polynomial ring $\mathbb{F}_q[X]$ play an important role. We establish the omega result by an analog of Landau's theorem \cite[Theorem 1.7]{Vaughan} concerning Dirichlet series with non-negative coefficients.

%\subsection{Organization}
The structure of the article is as follows: We discuss some preliminary results in Section 2. In Sections $3$ and $4$, we investigate the asymptotic behavior of the average, the omega result, and the Chebyshev's bias phenomenon for the analog of the irrational factor function of order $k$ over number fields and function fields, respectively.

\subsection{Notation}
 We write $f(x)=\BigO{g(x)}$ or equivalently $f(x)\ll g(x)$ if there exists a constant $C$ such that $|f(x)|\leq Cg(x)$ for all $x$. And $f(x)=\Omega_{\pm}(g(x))$ implies that $\limsup{f(x)}/{g(x)}>0$ and $\liminf{{f(x)}/{g(x)}}<0$. $\zeta(s)$ and $L(s,\chi)$ denote the Riemann zeta function and the Dirichlet L-function modulo $q$, respectively. Moreover, $(a,b)=1$ means that $a$ and $b$ are coprime.
 %and $\mathcal{L}(s,\chi^{\prime})$ is the Hecke L-function. Here, $\chi$ and $\chi^{\prime}$ are Dirichlet character and Hecke character, respectively. 

%We use $\phi$ to denote the Euler totient function, and $\mu$ to the M\"{o}bius function. $\zeta(s)$ is the Riemann zeta function, 

\subsection{Acknowledgements}
The author is grateful to Sneha Chaubey for suggesting this problem, for the generous discussions, and for insightful suggestions during the preparation of the article. The author also acknowledges the support from the UGC, Department of Higher Education, Government of India, under NTA Ref. no. 191620135578.

\section{Preliminaries}
In order to prove our main theorems, we require some preliminary results. Some of them can be obtained from previous works, while some of them require proof. We begin by deriving mean value estimates for $\zeta(s)/s$ and $L(s,\chi)/s$.

 %\begin{lem}\cite[Theorem 2, p. 132]{Tenenbaum}\label{parronlemma}
 %Let $F(s):=\sum_{n=1}^{\infty}f(n) {n}^{-s}$ be the Dirichlet series for the arithmetic function $f(n)$, with abscissa of convergence $\sigma_a$. If $\alpha>\max(0,\sigma_a)$, $T\ge 1$ and $x\ge 1$, then
 %\begin{center}
     %$\sum_{n\le x}^{\prime}f(n)=\frac{1}{2\pi i}\int_{\alpha-iT}^{\alpha+iT}F(s)\frac{x^s}{s}ds+R(T),$
 %\end{center}
 %\[\sum_{n\le x}f(n)=\frac{1}{2\pi i}\int_{\alpha-iT}^{\alpha+iT}F(s)\frac{x^s}{s}ds+R(T),\] 
 %where $\sum^{\prime}$ signifies that the coefficient of $f(n)$ is $\frac{1}{2}$ when $x \in \mathbb{N}$ and $R(T)\ll \frac{x^{\alpha}}{T}\sum_{n=1}^{\infty}\frac{|f(n)|}{n^{\alpha}|\log(x/n)|}.$ 
% \end{lem}

%The following bounds for the Riemann zeta function and the Dirichlet $L$-function modulo $q$ will play a crucial role in the applications of the Perron's formula.
\begin{prop}\label{prop2}
    If $0<\sigma<1/2$ is a real number, then 
    \[\int_0^T\frac{|\zeta(\sigma+it)|}{|\sigma+it|}dt\ll T^{\frac{1}{2}-\sigma}\log T.  \]
  \begin{proof}
      We begin with the asymptotic formula $(13)$ from \cite{Chahal}:
      \[\int_0^T\frac{|\zeta(\sigma+it)|}{|\sigma+it|}dt\ll 1+\log T\max_{0\leq n\leq\lfloor\log T\rfloor}\frac{1}{2^n}\int_{0}^{2^n}|\zeta(\sigma+it)|dt.\numberthis\label{I10} \]
The functional equation for Riemann zeta function is given by (see \cite{Davenport}, p. 59)
\[\zeta(s)=Y(s)\zeta(1-s), \]
where $Y(s)=\pi^{s-1/2}\Gamma((1-s)/2)/\Gamma(s/2)$. By Stirling's formula, we have
$Y(\sigma+it)\asymp t^{\frac{1}{2}-\sigma}$. Thus, the above identity gives
\[\zeta(\sigma+it)=\BigO{t^{\frac{1}{2}-\sigma}\zeta(1-\sigma+it)}. \numberthis\label{I9}\]
Employing Cauchy-Schwarz inequality and using the mean value formula \cite[Theorem 1.1]{Ivic} with \eqref{I9}, we obtain
     \[\int_{0}^{2^n}|\zeta(\sigma+it)|dt\ll 2^{\frac{n}{2}+n(1-\sigma)}. \]
     Inserting the above estimate into \eqref{I10} gives the required result.
  \end{proof}  
\end{prop}

\begin{prop}\label{prop3}
    If $0<\sigma<1/2$ is a real number, then 
    \[\int_0^T\frac{|L(\sigma+it, \chi)|}{|\sigma+it|}dt\ll_q T^{\frac{1}{2}-\sigma}\log T.  \]
\end{prop}
\begin{proof}
    The proof follows along the same lines as Proposition \ref{prop2}, using the functional equation for the Dirichlet $L$-function modulo $q$ and the mean value formula for Dirichlet series (see \cite{MR337847}, Chap. 6, p. 50).
    \end{proof}

Next, we state a classical theorem of Landau concerning the singularities of the Mellin transform of a non-negative function.

\begin{prop}\cite{LandauE}\label{Landau}
    Let $A(x)$ be a real-valued function in one variable, and $A(x)$ does not change its sign for $x>x_0$, where $x_0$ is a sufficiently large real number. Suppose also for some real number $\beta<\gamma$, that Mellin transform $g(s):=\int_1^{\infty}A(x)x^{-s-1}dx$ is analytic for $\mathcal{R}(s)>\gamma$, can be analytically continued to the real segment $(\beta,\gamma]$. Then $g(s)$ represents an analytic function in the half plane $\mathcal{R}(s)>\beta$.
\end{prop}
We also use the following result, which is an analogue of Landau's theorem \cite[Theorem 1.7]{Vaughan}, concerning Dirichlet series with non-negative coefficients.
\begin{prop}\cite[Theorem 15.1]{Vaughan}\label{omega}
    Suppose that $A(x)$ is a bounded Riemann integrable function in any finite interval $1\leq x\leq X$, and that $A(x)\geq 0$ for all $x>X_0$. Let $\sigma_c$ denote the infimum of those $\sigma$ for which $\int_{X_0}^{\infty}A(x)x^{-\sigma}dx<\infty$. Then the function
    \[F(s)=\int_{1}^{\infty}A(x)x^{-s}dx \]
    is analytic in the half plane $\sigma>\sigma_c$, but not at the point $s=\sigma_c$.
\end{prop}

\begin{prop}\cite[Lemma 2.8]{Roy}\label{Roy}
 Let $0<\lambda_1<\ldots<\lambda_n\to\infty$ be any sequence of real numbers, and let $\{a_n\}$ be a sequence of complex numbers. Let the Dirichlet series $F(s)=\sum_{n=1}^{\infty}\frac{a_n}{\lambda_n^s}$ be absolutely convergent for some $\Re(\sigma)>\sigma_a$. If $\sigma_0>\max(0,\sigma_a)$ and $x>0$, then
 \begin{center}
     $\sum_{\lambda_n\leq x}^{\prime}a_n=\frac{1}{2\pi i}\int_{\sigma_0-iT}^{\sigma_0+iT}F(s)\frac{x^s}{s}ds+R,$
 \end{center}
 where $\sum^{\prime}$ signifies that the coefficient of $a_n$ is $\frac{1}{2}$ when $x \in \mathbb{N}$ and
 \small{\[R\ll \sum_{\substack{\frac{x}{2}<\lambda_n<2x\\n\ne x}}|a_n|\min{\left(1, \frac{x}{T|x-\lambda_n|}\right)}+\frac{4^{\sigma_0}+x^{\sigma_0}}{T}\sum_{n=1}^{\infty}\frac{|a_n|}{\lambda_n^{\sigma_0}}.\]}
\end{prop}

 We now state an ingredient that will play the same role in deriving the asymptotic formula for the average of the irrational factor function over function fields as that played by Perron's formula in the classical setting.
\begin{prop}\label{prop7}
  Let $a(f)$ be an arithmetic function on $\mathbb{F}_q[X]$ and let $F(s)=\sum_{\substack{f\in\mathbb{F}_q[X] \\\ \text{monic}}}\frac{a(f)}{|f|^s}$ be a Dirichlet series with abscissa of convergence $\sigma_a$. If $\alpha>\max(0, \sigma_a),\ T\geq 1$ and $N\geq 1$, we have
  \begin{center}
      $\sum^{\prime}_{\deg(f)\leq N}a(f)=\frac{1}{2\pi i}\int_{\alpha-iT}^{\alpha+iT}F(s)\frac{q^{Ns}}{s}ds+\BigO{\frac{q^{\sigma N}}{T}\sum_{f}\frac{|a(f)|}{|f|^{\sigma}|\log(q^N/|f|)}}.$
  \end{center}
  
\end{prop}
\begin{proof}
    The proof follows from (\cite{Tenenbaum}, p. 218).
\end{proof}

We next recall the Phragm\'{e}n-Lindel\"{o}f principle.
\begin{thm}\cite[Theorem 5.53]{Iwaniec}\label{Phragmen}
Let $f(\sigma+it)$ be analytic in the strip $a\leq\sigma\leq b$ for some real numbers $a<b$, with $f(\sigma+it)\ll \exp{(\epsilon|t|)}$. If $|f(a+it)|\ll |t|^{c_1}$ and $|f(b+it)|\ll |t|^{c_2}$, then
\[|f(\sigma+it)|\ll |t|^{c(\sigma)}, \]
uniformly in $a\leq\sigma\leq b$, where $c(\sigma)$ is linear in $\sigma$ with $c(a)=c_1$ and $c(b)=c_2$.
\end{thm}

\subsection{Bounds for the Riemann zeta function and Dirichlet $L$-function modulo $q$}
 We also use the following standard bounds of $\zeta(s)$ \cite[page 47]{MR882550},
\begin{displaymath}
 \zeta(\sigma+it) \ll \left\{
   \begin{array}{lr}
     {t^{\frac{1- \sigma}{2}} \log t}, &\  0\le \sigma \le 1,\\
     { \log t},  &\ 1\le\sigma \le 2,\\
      {1} , &\ \sigma \ge 2.
     \end{array}
    \right.\numberthis\label{zetabound}
\end{displaymath}
For non principal characters $\chi\pmod{q}$, we use the following bounds of $L(s,\chi)$ (see \cite{MR551704})
\begin{displaymath}
 L(\sigma+it,\chi) \ll_q \left\{
   \begin{array}{lr}
    t^{\frac{35(1-\sigma)}{108}}\log^3 t , &\  1/2\le \sigma \le 1,\\
     { \log t},  &\ 1\le\sigma \le 2,\\
     {1} , &\ \sigma \ge 2.
     \end{array}
    \right.\numberthis\label{lbound}
\end{displaymath}
Employing \eqref{lbound} and functional equation for the Dirichlet $L$-function modulo $q$ (\cite{Davenport}, page 69-71), we obtain
\begin{displaymath}
 L(\sigma+it,\chi) \ll_q \left\{
   \begin{array}{lr}
   { t^{\frac{1}{2}-\sigma}\log t},  &\ -1\le\sigma \le 0,\\
    t^{\frac{54-73\sigma}{108}}\log^3 t , &\  0\leq\sigma\leq\frac{1}{2}.
     \end{array}
    \right.\numberthis\label{lbound2}
\end{displaymath}

\subsection{Bounds for Dedekind zeta function and Hecke $L$-function with GLH}
Employing Theorem \ref{Phragmen} and Generalized Lindel\"{o}f hypothesis, we obtain the following bounds for the Dedekind zeta function associated with a number field $\mathbb{K}$.
\begin{displaymath}
 \zeta_{\mathbb{K}}(\sigma+it) \ll \left\{
   \begin{array}{lr}
     {|t|^{\epsilon} }, &\  1/2\le \sigma \le 1,\\
     { \log t},  &\ 1\le\sigma \le 2,
     \end{array}
    \right.\numberthis\label{zetak}
    \end{displaymath}
    and for the Hecke $L$-function modulo $\mathfrak{q}$, we have
\begin{displaymath}
\mathcal{L}(\sigma+it, \chi) \ll_{\mathfrak{q}} \left\{
   \begin{array}{lr}
     {|t|^{\epsilon} }, &\  1/2\le \sigma \le 1,\\
     { \log t},  &\ 1\le\sigma \le 2,
     \end{array}
    \right.\numberthis\label{L_k}
    \end{displaymath}
    
\subsection{Bounds for Riemann zeta function and Dirichlet $L$-function over function fields}
Note that the Riemann hypothesis holds in function fields; consequently, the Generalized Lindel\"{o}f hypothesis is true. Therefore, using the functional equation for the Riemann zeta function over function fields \cite[p. 12]{Rosen} and the Generalized Lindel\"{o}f hypothesis in conjunction with Theorem \ref{Phragmen}, we derive the following bounds
\begin{displaymath}
 \zeta_{q}(\sigma+it) \ll \left\{
   \begin{array}{lr}
     {|t|^{\epsilon} }, &\  0\le \sigma \le 1,\\
     { C_1(q)},  &\ 1<\sigma \le 2,
     \end{array}
    \right.\numberthis\label{zetaq}
    \end{displaymath}
and employing Theorem \ref{Phragmen} with the functional equation of Dirichlet $L$-function modulo $\mathfrak{m}$ over the function fields, we obtain
\begin{displaymath}
 L_{q}(\sigma+it,\chi) \ll_{\mathfrak{m}} \left\{
   \begin{array}{lr}
     {|t|^{\epsilon} }, &\  0\le \sigma \le 1,\\
     { C_2(q)},  &\ 1<\sigma \le 2,
     \end{array}
    \right.\numberthis\label{L_q}
    \end{displaymath}
 where $C_1(q)$ and $C_2(q)$ are constants depending on $q$.

\section{ Number fields}
\subsection{Proof of Theorem \ref{thm6 Tauberian}}
We begin with the sum
\[\mathcal{S}_{k,\mathbb{K}}(x)=\sum_{\substack{\mathfrak{I}\subset\mathcal{O}_{\mathbb{K}}\\ \mathcal{N}(\mathfrak{I})\leq x}}\mathcal{F}_{k}(\mathfrak{I})=\sum_{n\leq x}a(n), \]
where $a(n)=\sum_{\substack{\mathfrak{I}\subset\mathcal{O}_{\mathbb{K}}\\ \mathcal{N}\mathfrak{I}=n}}\mathcal{F}_{k}(\mathfrak{I})$.
 Note that the arithmetic function $a(n)$ is multiplicative.
The Dirichlet series of $a(n)$ is given by:
\[F(s)=\sum_{n=1}^{\infty}\frac{a(n)}{n^s}= \frac{\zeta_{\mathbb{K}}(s-1)}{\zeta_{\mathbb{K}}(ks-k)}R_k(s), \]
where $R_k(s)=\prod_{\mathfrak{p}}(1+A_{\mathfrak{p}}(s))$ and
\[A_{\mathfrak{p}}(s)=\frac{\frac{1}{(\mathcal{N}\mathfrak{p})^{ks-\frac{1}{k}}}\sum_{m=0}^{\infty}\frac{1}{(\mathcal{N}\mathfrak{p})^{ms+\frac{1}{k}-\frac{1}{k+m}}}}{\left(1+\frac{1}{(\mathcal{N}\mathfrak{p})^{s-1}}+\frac{1}{(\mathcal{N}\mathfrak{p})^{2s-2}}+\cdots+\frac{1}{(\mathcal{N}\mathfrak{p})^{(k-1)(s-1)}} \right)}\numberthis\label{I32} .\]
We observe that
\begin{align*}
    \left|1+\frac{1}{(\mathcal{N}\mathfrak{p})^{s-1}}+\frac{1}{(\mathcal{N}\mathfrak{p})^{2s-2}}+\cdots+\frac{1}{(\mathcal{N}\mathfrak{p})^{(k-1)(s-1)}} \right|
    %=\left|\frac{1-\frac{\chi(p^k)}{p^{k(s-1)}}}{1-\frac{\chi(p)}{p^{s-1}}}\right|
    \geq \frac{1-\frac{1}{(\mathcal{N}\mathfrak{p})^{k(\sigma-1)}}}{1+\frac{1}{(\mathcal{N}\mathfrak{p})^{\sigma-1}}}\ \text{and}\ 
    \left|\sum_{m=0}^{\infty}\frac{(\mathcal{N}\mathfrak{p})^{-ks+\frac{1}{k}}}{(\mathcal{N}\mathfrak{p})^{ms+\frac{1}{k}-\frac{1}{k+m}}}\right|
     %\leq \frac{1}{p^{k\sigma-\frac{1}{k}}}\sum_{m=0}^{\infty}\frac{1}{p^{m\sigma+\frac{1}{k}-\frac{1}{k+m}}}
     %\leq \frac{1}{p^{k\sigma-\frac{1}{k}}}\sum_{m=0}^{\infty}\frac{1}{p^{m\sigma}}
     \leq \frac{2}{(\mathcal{N}\mathfrak{p})^{k\sigma-\frac{1}{k}}}.
\end{align*}
The above two estimates with \eqref{I32} yield
%\[|A_p(s)|\leq \frac{c}{p^{k\sigma-\frac{1}{k}}},\]  gives 
\[\sum_{\mathfrak{p}\subset\mathcal{O}_{\mathbb{K}}}|A_{\mathfrak{p}}(s)|\leq \sum_{\mathfrak{p}\subset\mathcal{O}_{\mathbb{K}}}\frac{c}{(\mathcal{N}\mathfrak{p})^{k\sigma-\frac{1}{k}}}<\infty\ \text{for}\ \sigma>\frac{1}{k}+\frac{1}{k^2}. \numberthis\label{I8}\]

Therefore, $R_k(s)$ is absolutely convergent and defines an analytic function in the half-plane $\Re(s)>1$. It is well known that the Dedekind zeta function has analytic continuation to the complex plane except for a simple pole at $s=1$. Also, $\zeta_{\mathbb{K}}(ks-k)$ has no zeros in the half-plane $\Re(s)>1+\frac{1}{k}$ (see \cite{Habiba}). Thus, the Dirichlet series $F(s)$ is absolutely convergent for $\Re(s)>2$ and extends to a meromorphic function in the half-plane $\Re(s)\geq 2$ except for a simple pole at $s=2$ with residue ${\lambda_{\mathbb{K}}{R}_{k}(2)}/{2\zeta_{\mathbb{K}}(k)}$. Moreover, the coefficients $a(n)$ of the Dirichlet series $F(s)$ are non-negative, thus employing \cite[Theorem 1.2]{Murty}, we obtain
\[\mathcal{S}_{k,\mathbb{K}}(x)=\frac{x^2\lambda_{\mathbb{K}}{R}_{k}(2)}{2\zeta_{\mathbb{K}}(k)}+\BigO{\frac{x^2}{\log x}}.\numberthis\label{s2} \]
This completes the proof of Theorem \ref{thm6 Tauberian}.

\subsection{Proof of Theorem \ref{thm4}}
 We begin with the sum in \eqref{Ns1} and use the orthogonality relation of Hecke characters modulo ${\mathfrak{q}}$. Therefore
 \begin{align*}
     \mathcal{S}_{k,\mathbb{K}}(x;\mathfrak{q},\mathfrak{a})&=\sum_{\substack{\mathfrak{I}\subset\mathcal{O}_{\mathbb{K}}\\ \mathcal{N}(\mathfrak{I})\leq x\\ \mathfrak{I}\equiv \mathfrak{a}\pmod{\mathfrak{q}}}}\mathcal{F}_{k}(\mathfrak{I})
=\sum_{\substack{\mathfrak{I}\subset\mathcal{O}_{\mathbb{K}}\\ \mathcal{N}(\mathfrak{I})\leq x}}\mathcal{F}_{k}(\mathfrak{I})\frac{1}{|(\mathcal{O}_{\mathbb{K}}/ \mathfrak{q})^*|}\sum_{\chi\pmod{\mathfrak{q}}}\bar{\chi}(\mathfrak{a})\chi(\mathfrak{I})\\
&=\frac{1}{|(\mathcal{O}_{\mathbb{K}}/ \mathfrak{q})^*|}\sum_{\chi\pmod{\mathfrak{q}}}\bar{\chi}(\mathfrak{a})\sum_{\substack{\mathfrak{I}\subset\mathcal{O}_{\mathbb{K}}\\ \mathcal{N}(\mathfrak{I})\leq x}}\chi(\mathfrak{I})\mathcal{F}_{k}(\mathfrak{I}).\numberthis\label{s1}
 \end{align*}
 We first observe that $\chi(\mathfrak{I})\mathcal{F}_{k}(\mathfrak{I})$ is multiplicative function thus the Dirichlet series of $\chi(\mathfrak{I})\mathcal{F}_{k}(\mathfrak{I})$ is given by:
 \begin{align*}
F(s)&=\sum_{\mathfrak{I}}\frac{\chi(\mathfrak{I})\mathcal{F}_{k}(\mathfrak{I})}{(\mathcal{N}\mathfrak{I})^s}
%=\prod_{\mathfrak{p}}\left(1+\frac{\chi(\mathfrak{p})\mathcal{N}\mathfrak{p}}{(\mathcal{N}\mathfrak{p})^s}+\cdots+\frac{\chi(\mathfrak{p}^{k-1})(\mathcal{N}\mathfrak{p})^{k-1}}{(\mathcal{N}\mathfrak{p})^{(k-1)s}}+\frac{\chi(\mathfrak{p}^{k})(\mathcal{N}\mathfrak{p})^{\frac{1}{k}}}{(\mathcal{N}\mathfrak{p})^{ks}}+\cdots \right)\\
%&=\prod_{\mathfrak{p}}\left(1+\frac{\chi(\mathfrak{p})}{(\mathcal{N}\mathfrak{p})^{s-1}}+\cdots+\frac{\chi(\mathfrak{p}^{k-1})}{(\mathcal{N}\mathfrak{p})^{(k-1)(s-1)}}+\frac{\chi(\mathfrak{p}^{k})}{(\mathcal{N}\mathfrak{p})^{ks-\frac{1}{k}}}+\cdots \right)
=\frac{\mathcal{L}(s-1,\chi)}{\mathcal{L}(ks-k,\chi^k)}R_k(s,\chi), \numberthis\label{N6}
 \end{align*}
where $R_k(s,\chi)=\prod_{\mathfrak{p}}(1+\mathfrak{A}_{\mathfrak{p}, \chi}(s))$ and
\[\mathfrak{A}_{\mathfrak{p},\chi}(s)=\frac{\frac{1}{(\mathcal{N}\mathfrak{p})^{ks-\frac{1}{k}}}\sum_{m=0}^{\infty}\frac{\chi(\mathfrak{p}^{k+m})}{(\mathcal{N}\mathfrak{p})^{ms+\frac{1}{k}-\frac{1}{k+m}}}}{\left(1+\frac{\chi(\mathfrak{p})}{(\mathcal{N}\mathfrak{p})^{s-1}}+\frac{\chi(\mathfrak{p}^2)}{(\mathcal{N}\mathfrak{p})^{2s-2}}+\cdots+\frac{\chi(\mathfrak{p}^{k-1})}{(\mathcal{N}\mathfrak{p})^{(k-1)(s-1)}} \right)} .\]
Note that $R_k(s,\chi)$ is absolutely convergent for $\Re(s)>1$ as in \eqref{I8}. For $\chi=\chi_0$, The Hecke $\mathcal{L}$-function $\mathcal{L}(s-1,\chi)$ admits an analytic continuation to the complex plane except for a simple pole at $s=2$ when $\chi=\chi_0$. Moreover, $\mathcal{L}(ks-k,\chi^k)$ has no zeros in the half-plane $\Re(s)>1+1/k$ (see \cite{Coleman}). Thus, the Dirichlet series $F(s)$ has meromorphic continuation in the region $\Re(s)>1+1/k$ with a simple pole at $s=2$. We use Proposition \ref{Roy} with $a_n=A(n,\chi)=\sum_{\substack{\mathfrak{I}\subset\mathcal{O}_{\mathbb{K}}\\ \mathcal{N}\mathfrak{I}=n}}\chi(\mathfrak{I})\mathcal{F}_k(\mathfrak{I})$ and $\alpha=2+\epsilon$, and obtain
\[\sum_{\substack{\mathfrak{I}\subset\mathcal{O}_K\\ \mathcal{N}\mathfrak{I}\leq x}}\chi(\mathfrak{I})\mathcal{F}_k(\mathfrak{I})=\sum_{n\leq x}A(n,\chi)=\frac{1}{2\pi i}\int_{\alpha-iT}^{\alpha+iT}F(s)\frac{x^s}{s}ds+E(T,\chi), \numberthis\label{int}\]
where
\[E(T,\chi)\ll \sum_{\substack{\frac{x}{2}<n<2x\\n\ne x}}|A(n,\chi)|\min{\left(1, \frac{x}{T|x-n|}\right)}+\frac{x^{\alpha}}{T}\sum_{n=1}^{\infty}\frac{|A(n,\chi)|}{n^{\alpha}}. \]
By dividing the interval $x/2<n<2x$ into two parts and using the following inequality
\[|A(n,\chi)|\leq\sum_{\substack{\mathfrak{I}\subset\mathcal{O}_{\mathbb{K}}\\ \mathcal{N}\mathfrak{I}=n}}\mathcal{F}_k(\mathfrak{I})\leq\sum_{\substack{\mathfrak{I}\subset\mathcal{O}_{\mathbb{K}}\\ \mathcal{N}\mathfrak{I}=n}}\mathcal{N}\mathfrak{I}\leq n\sum_{\substack{\mathfrak{I}\subset\mathcal{O}_{\mathbb{K}}\\ \mathcal{N}\mathfrak{I}=n}}1\leq n^{1+\epsilon}, \]
the error term in \eqref{int} can be estimated as
\begin{align*}
    E(T,\chi)&\ll \sum_{\substack{|n-x|<x^{\frac{1}{2}}\\n\ne x}}|A(n,\chi)|\min{\left(1, \frac{x}{T|x-n|}\right)}+\sum_{\substack{x+x^{\frac{1}{2}}<n<2x}}|A(n,\chi)|\min{\left(1, \frac{x}{T|x-n|}\right)}\\&+\frac{x^{\alpha}}{T}\sum_{n=1}^{\infty}\frac{n^{1+\epsilon}}{n^{\alpha}}
    \ll x^{1+\epsilon}+\frac{x^{2+\epsilon}\log x}{T}+\frac{x^{\alpha}}{T}.
\end{align*}
Combining the above estimate with \eqref{int}, we find that
\[\sum_{\substack{\mathfrak{I}\subset\mathcal{O}_K\\ \mathcal{N}\mathfrak{I}\leq x}}\chi(\mathfrak{I})\mathcal{F}_k(\mathfrak{I})=\frac{1}{2\pi i}\int_{\alpha-iT}^{\alpha+iT}\frac{\mathcal{L}(s-1,\chi)x^s}{s\mathcal{L}(ks-k,\chi^k)}R_k(s,\chi)ds+\BigO{x^{1+\epsilon}+\frac{x^{2+\epsilon}\log x}{T}+\frac{x^{\alpha}}{T}} .\numberthis\label{N3}\]
The integral on the right hand-side of the above identity is estimated separately in two case.

\textbf{Estimation of integral (Unconditionally):}
We deal with the cases $m=1$ and $m\geq 2$ separately. For $m=1$, to estimate the integral in \eqref{int}, we use the Vinogradov-Korobov zero-free region for the Dirichlet $L$-function $L(s,\chi)$ modulo $q$  \cite[Theorem 1.1]{Khale}, in conjunction with the mean value estimates for the Riemann zeta function and the Dirichlet $L$-function in terms of Proposition \ref{prop2} and \ref{prop3}, respectively. We choose $\alpha=2+1/\log x,\ \beta=1+1/k-c/(\log x)^{2/3}(\log\log x)^{1/3}$ and shift the line of integration into a rectangular contour with line segments joining the points $\beta\pm iT$ and $\alpha\pm iT$. We first consider the principal character $\chi=\chi_0$. By Cauchy's residue theorem, the integral on the right-hand side of \eqref{N3} can be expressed as
\[\frac{1}{2\pi i}\int_{\alpha-iT}^{\alpha+iT}\frac{x^sL(s-1,\chi_0)}{sL(ks-k,\chi_0)}M_{k,\chi_0}(s)ds=\frac{x^2M_{k,\chi_0}(2)}{2L(k,\chi_0)}\prod_{p|q}\left(1-\frac{1}{p} \right)+\sum_{j=1}^3I_j, \]
where $M_{k,\chi}(s)=\prod_{p}(1+A_{p,\chi}(s))$ and
$A_{p,\chi}(s)={{\left(1+\frac{\chi(p)}{p^{s-1}}+\cdots+\frac{\chi(p^{k-1})}{p^{(k-1)(s-1)}} \right)^{-1}}\sum_{m=0}^{\infty}\frac{\chi(p^{k+m})}{p^{(k+m)s-\frac{1}{k+m}}}}$. Moreover, $I_1$ and $I_3$ are integrals along the horizontal segments $[\alpha-iT,\beta-iT]$ and $[\alpha+iT, \beta+iT]$, respectively and $I_2$ is the integral along the vertical segment $[\beta-iT, \beta+iT]$. In order to estimate the integrals $I_j$'s, we use the bounds provided in \eqref{zetabound}, modulo multiplication by constant depending on $q$ and $k$. Therefore,
\begin{align*}
    I_1, I_3& %\int_{\beta}^{\alpha}\frac{Q^{\sigma}|\zeta(\sigma-1+iT)|}{|\sigma+iT||\zeta(k\sigma-k+ikT)|}|M_{k,\chi}(\sigma+iT)|d\sigma\\
    \ll_{q,k} \log T \left(\int_{\beta}^2\frac{x^{\sigma}|\zeta(\sigma-1+iT)|}{|\sigma+iT|}d\sigma + \int_2^{\alpha}\frac{x^{\sigma}|\zeta(\sigma-1+iT)|}{|\sigma+iT|}d\sigma\right)
    %&\ll_{q,k} \frac{(\log T)^2}{T}\left(\int_{\beta}^2x^{\sigma}T^{1-\frac{\sigma}{2}}d\sigma + \int_2^{\alpha}x^{\sigma}d\sigma\right)
    \ll_{q,k} \frac{x^2(\log T)^2}{T\log x}.
\end{align*}
In here, we used the bound $\zeta(k\sigma-k+iT)\gg1/\log T$ (see \cite{Vaughan}). Next, we estimate the integral $I_2$ using Proposition \ref{prop2}
\begin{align*}
    I_2&\ll_{q,k} x^{\beta}\int_0^T\frac{|\zeta(\beta-1+it)|}{|\beta+it||\zeta(k\beta-k+ikt)|}dt
    \ll_{q,k} x^{\beta}\log T\int_0^T\frac{|\zeta(\beta-1+it)|}{|\beta+it|}dt\ll_{q,k}x^{\beta}T^{\frac{3}{2}-\beta}(\log T)^2.
\end{align*}
%The bounds established by Dong and Meng will be applicable to our case, with the exception of multiplication by a constant that depends on $q$. Therefore,
%\[I_1, I_3\ll_q \frac{Q^2(\log Q)^2}{T} \] and
%\[I_2\ll Q^{1+\frac{1}{k}}T^{\frac{1}{2}-\frac{1}{k}}\log Q\log T .\]
Next, we consider the case for non-principal character $\chi$ modulo $q$. We continue with the contour defined above and use the bounds provided in \eqref{lbound}. Therefore
\begin{align*}
    I_1, I_3&\ll_{q,k} \log T\int_{\beta}^{\alpha}\frac{x^{\sigma}|L(\sigma-1+iT,\chi)|}{|\sigma+iT|}d\sigma\\
    &\ll_{q,k}\frac{(\log T)^2}{T}\left((\log T)^2\int_{\beta}^{\frac{3}{2}}{x^{\sigma}T^{\frac{127-73\sigma}{108}}}d\sigma+(\log T)^2\int_{\frac{3}{2}}^2x^{\sigma}T^{\frac{35(2-\sigma)}{108}}d\sigma+\int_{2}^{\alpha}{x^{\sigma}}d\sigma \right)
    \ll_{q,k}\frac{x^{2}(\log T)^{2}}{T\log x},
    \end{align*}
and
\begin{align*}
    I_2&\ll_{q,k}  x^{\beta}\log T\int_0^T\frac{|L(\beta-1+it,\chi)|}{|\beta-1+it|}dt
    \ll_{q,k} x^{\beta}T^{\frac{3}{2}-\beta}(\log T)^2,
\end{align*}
where we used $|L(k\sigma-k+it,\chi^k)|\gg_q 1/\log T$ (see \cite{Vaughan}). We take optimally
\[T=x^{\frac{2(k-1)}{3k-2}}\exp{\left(c^{\prime}\left(\frac{\log x}{\log\log x}\right)^{1/3} \right)}, \] 
where $c^{\prime}=2ck/(3k-2)$ and collect all the above estimates to obtain
%\[\sum_{n\leq Q}\chi_0(n)I_k(n)=\frac{Q^2}{2L(2,\chi_0)}M_{k,\chi_0}(2)\prod_{p|q}\left(1-\frac{1}{p} \right)+\BigOkq{Q^{\frac{2(2k-1)}{3k-2}}\exp{\left(-c^{\prime}\left(\frac{\log Q}{\log\log Q}\right)^{1/3}+\log(\log Q)^2 \right)}},\numberthis\label{I_3} \]
\[\sum_{n\leq x}\chi_0(n)I_k(n)=\frac{x^2M_{k,\chi_0}(2)}{2L(k,\chi_0)}\prod_{p|q}\left(1-\frac{1}{p} \right)+\BigOkq{x^{\frac{2(2k-1)}{3k-2}}\exp{\left(-c\left(\frac{\log x}{\log\log x}\right)^{1/3} \right)}}.\numberthis\label{I_3} \]
For $\chi\ne\chi_0$, the main term in the above asymptotic formula disappears. Hence, inserting \eqref{I_3} into \eqref{s1} completes the proof for $m=1$.

For $m\geq 2$, to estimate the integral in \eqref{N3}, we deform the line of integration into a rectangular contour with vertices $\beta\pm it$ and $\alpha\pm it$, where $\alpha=2+\epsilon$ and $\beta=1+\frac{1}{k}+\frac{1}{\log x}$. 

\textbf{Case (i):} For $\chi=\chi_0$, the integrand in \eqref{N3} has a simple pole at $s=2$. By Cauchy's residue theorem, we have
\begin{align*}
   \frac{1}{2\pi i}\int_{\alpha-iT}^{\alpha+iT}\frac{x^s\mathcal{L}(s-1,\chi_0)}{s\mathcal{L}(ks-k,\chi_0)}R_k(s,\chi_0)ds&=\frac{x^2\lambda_{\mathbb{K}}R_k(2,\chi_0)}{2\mathcal{L}(k,\chi_0)}\prod_{\mathfrak{p}|\mathfrak{q}}\left(1-\frac{1}{\mathcal{N}\mathfrak{p}} \right)+\frac{1}{2\pi i}\left(\int_{\alpha-iT}^{\beta-iT}+\int_{\beta-iT}^{\beta+iT}+\int_{\beta+iT}^{\alpha+iT} \right)\\ &\times\frac{x^s\mathcal{L}(s-1,\chi_0)R_k(2,\chi_0)}{s\mathcal{L}(ks-k,\chi_0)}ds=\frac{x^2\lambda_{\mathbb{K}}R_k(2,\chi_0)}{2\mathcal{L}(k,\chi_0)}\prod_{\mathfrak{p}|\mathfrak{q}}\left(1-\frac{1}{\mathcal{N}\mathfrak{p}} \right)+\sum_{i=1}^3J_i. \numberthis\label{N10} 
\end{align*}
%\[\frac{1}{2\pi i}\int_{\alpha-iT}^{\alpha+iT}\frac{x^s\mathcal{L}(s-1,\chi)}{s\mathcal{L}(ks-k,\chi^k)}R_k(s,\chi)ds=\frac{x^2\lambda_KR_k(s,\chi)}{2\mathcal{L}(k,\chi_0)}\prod_{\mathfrak{p}|\mathfrak{q}}\left(1-\frac{1}{\mathcal{N}\mathfrak{p}} \right)+\sum_{j=1}^3I_j , \]
In order to estimate the integrals in the above identity, we use the following result on the bounds of the Dedekind zeta function.
\begin{thm}\cite[Theorem 4]{Rademacher}\label{Rade}
    In the strip $-\eta\leq\sigma\leq 1+\eta,\ 0<\eta\leq \frac{1}{2}$, the Dedekind zeta function $\zeta_{\mathbb{K}}$ belonging to the number field $\mathbb{K}$ of degree $m$ and discriminant $d$ satisfies the inequality
    \[|\zeta_{\mathbb{K}}(s)|\leq 3\left|\frac{1+s}{1-s} \right|\left(|d|\left(\frac{|1+s|}{2\pi} \right)^m \right)^{\frac{1+\eta-\sigma}{2}}\zeta(1+\eta)^m. \]
\end{thm}
Employing the above theorem with $\eta=\epsilon$ and using the Euler product representation for the Dedekind zeta function, we have
\[\zeta_{\mathbb{K}}\left(\frac{1}{2}+it\right)\ll |d|^{\frac{1}{4}+\epsilon}t^{\frac{m}{4}+\epsilon}(\log t)^m\ll_d t^{\frac{m}{4}+\epsilon}\ \text{and}\ \frac{1}{\zeta_{\mathbb{K}}(\frac{k}{2}+ikt)}\leq \frac{\zeta_{\mathbb{K}}(k/2)}{\zeta_{\mathbb{K}}(k)}\ll t^{\epsilon}. \]
Using the above bounds, the integral $J_2$ is estimated as follows:
\[J_2\ll_{\mathfrak{q},k, d}x^{\beta}\int_{0}^{T}t^{\frac{m(k-1)}{2k}-1+\epsilon}dt\ll_{\mathfrak{q},k,d}x^{\beta}T^{\frac{m(k-1)}{2k}+\epsilon}.\numberthis\label{N11} \]
Next, we estimate the integral $J_3$ using Theorem \ref{Rade}
\begin{align*}
     J_3&\ll_{\mathfrak{q},k,d} \int_{\beta}^{\alpha}\frac{x^{\sigma}|\zeta_{\mathbb{K}}(\sigma-1+iT)|}{|\sigma+iT||\zeta_{\mathbb{K}}(k\sigma-k+ikT)|}d\sigma\\ 
    &\ll_{\mathfrak{q}, k, d}\frac{1}{T^{1-\epsilon}}\left(\int_{\beta}^{2}T^{(2+\epsilon-\sigma)\frac{m}{2}}x^{\sigma}d\sigma+\log T\int_{2}^{\alpha}x^{\sigma}d\sigma \right)\ll_{\mathfrak{q}, k, d}\frac{x^2}{(\log x) T^{1-\frac{m(k-1)}{2k}+\epsilon}}.\numberthis\label{N12}
\end{align*}
Similar estimate follow for $J_1$.
We collect the above estimates from \eqref{N11}, \eqref{N12} and substitute into \eqref{N10} with $T=x^\frac{k-1}{k}$ and obtain
\[\frac{1}{2\pi i}\int_{\alpha-iT}^{\alpha+iT}\frac{x^s\mathcal{L}(s-1,\chi)}{s\mathcal{L}(ks-k,\chi^k)}R_k(s,\chi)ds=\frac{x^2\lambda_{\mathbb{K}}R_k(2,\chi_0)}{2\mathcal{L}(k,\chi_0)}\prod_{\mathfrak{p}|\mathfrak{q}}\left(1-\frac{1}{\mathcal{N}\mathfrak{p}} \right)+\BigOqkd{x^{1+\frac{1}{k}+\frac{m(k-1)^2}{2k^2}+\epsilon}}. \numberthis\label{N1}\]

\textbf{Case (ii):} For non-principal character $\chi\ne\chi_0$, we proceed with the same rectangular contour defined above to estimate the integral in \eqref{N3}.
By Cauchy's residue theorem, we have
\begin{align*}
   \frac{1}{2\pi i}\int_{\alpha-iT}^{\alpha+iT}\frac{x^s\mathcal{L}(s-1,\chi)}{s\mathcal{L}(ks-k,\chi^k)}R_k(s,\chi)ds&=\sum_{i=1}^3J_i, \numberthis\label{27}
\end{align*}
where $J_i$ are integrals as in \eqref{N10}. 
 The following result is useful for the bounds of Hecke $L$-function modulo $\mathfrak{q}$. %to estimate the integrals in \eqref{N10} for non-principal character $\chi\ne\chi_0$.
\begin{thm}\cite[Lemma 4]{Fogels}
  Let $\mathbb{K}$ be a number field of degree $m$ and discriminant $d$. Let $\chi$ be a non-principal Hecke character modulo $\mathfrak{q}$. For any positive $\delta\leq \frac{1}{\log d}<\frac{1}{2}$, we have uniformly in $-\delta\leq\sigma\leq 1+\delta$
    \[\mathcal{L}(s,\chi)\ll \delta^{-m}d^{\frac{1-\sigma}{2}}(1+|t|)^{\frac{(1+\delta-\sigma)m}{2}}. \]
\end{thm}
We apply above theorem to estimate the integrals in \eqref{27}. Therefore,
\[J_2\ll_{\mathfrak{q},k,d}x^{\beta}\int_{0}^{T}\frac{x^{\sigma}|\mathcal{L}(\sigma-1+iT,\chi)|}{|\sigma+iT||\mathcal{L}(k\sigma-k+ikT,\chi^k)|}d\sigma\ll_{\mathfrak{q},k,d}x^{\beta}\int_{0}^{T}t^{\frac{m(k-1)}{2k}-1+\epsilon}dt\ll_{\mathfrak{q},k,d}x^{\beta}T^{\frac{m(k-1)}{2k}+\epsilon}, \]
and
\[ J_3\ll_{\mathfrak{q},k,d} \int_{\beta}^{\alpha}\frac{x^{\sigma}|\mathcal{L}(\sigma-1+iT,\chi)|}{|\sigma+iT||\mathcal{L}(k\sigma-k+ikT,\chi^k)|}d\sigma
    \ll_{\mathfrak{q},k,d}\frac{1}{T^{1-\epsilon}}\int_{\beta}^{\alpha}T^{(2+\epsilon-\sigma)\frac{m}{2}}x^{\sigma}d\sigma\ll_{\mathfrak{q},k,d}\frac{x^\alpha}{T^{1-\frac{m(k-1)}{2k}-\epsilon}\log x}. \]
 A similar estimates holds for $J_1$. Collecting above estimate with $T=x^{\frac{k-1}{k}}$, we obtain
  \[\frac{1}{2\pi i}\int_{\alpha-iT}^{\alpha+iT}\frac{x^s\mathcal{L}(s-1,\chi)}{s\mathcal{L}(ks-k,\chi^k)}R_k(s,\chi)ds=\BigOqkd{x^{1+\frac{1}{k}+\frac{m(k-1)^2}{2k^2}+\epsilon}}. \numberthis\label{N2}\]  
Collecting the estimates \eqref{N3}, \eqref{N1}, and \eqref{N2}, with $T=x^{\frac{1}{2}}$ and inserting them into \eqref{s1} gives the required result unconditionally.

\textbf{Estimation of integral under the Generalized Lindel\"{o}f hypothesis:} We proceed with the same rectangular contour as defined above.

\noindent
\textbf{Case (i):} For $\chi=\chi_0$, the integrand in \eqref{int} has a simple pole at $s=2$. By Cauchy's residue theorem and \eqref{N10}, we have
\begin{align*}
   \frac{1}{2\pi i}\int_{\alpha-iT}^{\alpha+iT}\frac{x^s\mathcal{L}(s-1,\chi_0)}{s\mathcal{L}(ks-k,\chi_0)}R_k(s,\chi_0)ds&=\frac{x^2\lambda_{\mathbb{K}}R_k(2,\chi_0)}{2\mathcal{L}(k,\chi_0)}\prod_{\mathfrak{p}|\mathfrak{q}}\left(1-\frac{1}{\mathcal{N}\mathfrak{p}} \right)+\sum_{i=1}^3J_i, \numberthis\label{N23}
\end{align*}
where $J_i$ denotes the same integral as in \eqref{N10}. Employing the bounds in \eqref{zetak}, we have
\[J_2\ll_{\mathfrak{q},k}x^{\beta}\int_{0}^{T}t^{-1+\epsilon}dt\ll_{\mathfrak{q},k}x^{\beta}T^{\epsilon}\ \text{and}\ J_1,J_3 
    \ll_{\mathfrak{q},k}\frac{1}{T^{1-\epsilon}}\left(T^{\epsilon}\int_{\beta}^{2}x^{\sigma}d\sigma+\log T\int_{2}^{\alpha}x^{\sigma}d\sigma \right)\ll_{\mathfrak{q},k}\frac{x^2}{T^{1-\epsilon}}. \]
%and
%\begin{align*}
   % J_3 &\ll_{k,\mathfrak{q}}\frac{1}{T^{1-\epsilon}}\left(T^{\epsilon}\int_{\beta}^{2}x^{\sigma}d\sigma+\log T\int_{2}^{\alpha}x^{\sigma}d\sigma \right)\ll_{k,\mathfrak{q}}\frac{x^2}{T^{1-\epsilon}}.
%\end{align*}
 By collecting the above estimates and substituting them into \eqref{N23} with $T=x^{1-\frac{1}{k}}$, we get
\[\frac{1}{2\pi i}\int_{\alpha-iT}^{\alpha+iT}\frac{x^s\mathcal{L}(s-1,\chi_0)}{s\mathcal{L}(ks-k,\chi_0)}R_k(s,\chi_0)ds=\frac{x^2\lambda_{\mathbb{K}}R_k(2,\chi_0)}{2\mathcal{L}(k,\chi_0)}\prod_{\mathfrak{p}|\mathfrak{q}}\left(1-\frac{1}{\mathcal{N}\mathfrak{p}} \right)+\BigOqk{x^{\frac{k+1}{k}+\epsilon}}.\numberthis\label{N24} \]

\textbf{Case (ii):} For $\chi\ne\chi_0$, by Cauchy's residue theorem, we have
\begin{align*}
   \frac{1}{2\pi i}\int_{\alpha-iT}^{\alpha+iT}\frac{x^s\mathcal{L}(s-1,\chi)}{s\mathcal{L}(ks-k,\chi^k)}R_k(s,\chi)ds&=\sum_{i=1}^3J_i, \numberthis
\end{align*}
where $J_i$ are same integrals as in \eqref{N10}. We use the bounds in \eqref{L_k} to obtain
\[J_2\ll_{\mathfrak{q},k}x^{\beta}\int_{0}^{T}t^{-1+\epsilon}dt\ll_{\mathfrak{q},k}x^{\beta}T^{\epsilon}\ \text{and}\ J_1,J_3 
    \ll_{\mathfrak{q},k}\frac{1}{T^{1-\epsilon}}\left(T^{\epsilon}\int_{\beta}^{2}x^{\sigma}d\sigma+\log T\int_{2}^{\alpha}x^{\sigma}d\sigma \right)\ll_{\mathfrak{q},k}\frac{x^2}{T^{1-\epsilon}}. \]
%and
%\begin{align*}
   % J_3 \ll_{k,\mathfrak{q}}\frac{1}{T^{1-\epsilon}}\left(T^{\epsilon}\int_{\beta}^{2}x^{\sigma}d\sigma+\log T\int_{2}^{\alpha}x^{\sigma}d\sigma \right)\ll_{k,\mathfrak{q}}\frac{x^2}{T^{1-\epsilon}}.
%\end{align*}
 We gather the above estimates and inserting them into \eqref{N23} and choose $T=x^{\frac{k-1}{k}}$. Therefore
\[\frac{1}{2\pi i}\int_{\alpha-iT}^{\alpha+iT}\frac{x^s\mathcal{L}(s-1,\chi)}{s\mathcal{L}(ks-k,\chi^k)}R_k(s,\chi)ds=\BigO{x^{\frac{k+1}{k}+\epsilon}}.\numberthis\label{N25} \]
By collecting \eqref{s1}, \eqref{N3}, \eqref{N24}, and \eqref{N25}, and setting $T=x^{\frac{k-1}{k}}$, we obtain the required result under GLH. This completes the proof of Theorem \ref{thm4}.

\subsection{Proof of Theorem \ref{thm6 Tauberian2}}

Using \eqref{s1}, we have
\begin{align*}
    \mathcal{S}_{k,\mathbb{K}}(x;\mathfrak{q},\mathfrak{a})&=\frac{1}{|(\mathcal{O}_{\mathbb{K}}/ \mathfrak{q})^*|}\sum_{\chi\pmod{\mathfrak{q}}}\bar{\chi}(\mathfrak{a})\sum_{\substack{\mathfrak{I}\subset\mathcal{O}_{\mathbb{K}}\\ \mathcal{N}(\mathfrak{I})\leq x}}\chi(\mathfrak{I})\mathcal{F}_{k}(\mathfrak{I})\ll_{\mathfrak{q}} \sum_{\substack{\mathfrak{I}\subset\mathcal{O}_{\mathbb{K}}\\ \mathcal{N}(\mathfrak{I})\leq x}}\mathcal{F}_{k}(\mathfrak{I}).
\end{align*}
The above estimate, in conjunction with \eqref{s2}, yields
\[\mathcal{S}_{k,\mathbb{K}}(x;\mathfrak{q},\mathfrak{a})\ll_{\mathfrak{q}} \frac{x^2\lambda_{\mathbb{K}}{R}_{k}(2)}{2\zeta_{\mathbb{K}}(k)}. \]
This completes the proof of Theorem \ref{thm6 Tauberian2}.

\subsection{Proof of Theorem \ref{thm5}}
In proving Theorem \ref{thm5}, we use an analog of Landau's result \cite[Theorem 1.7]{Vaughan} in terms of Proposition \ref{omega}.
We denote
\[\mathfrak{A}_{\mathbb{K}}(x)=-\mathcal{S}_{k,\mathbb{K}}(x;\mathfrak{q},\mathfrak{a})+\frac{x^2\lambda_{\mathbb{K}}{R}_{k}(2,\chi_0)}{2|(\mathcal{O}_\mathbb{K}/\mathfrak{q})^*|\mathcal{L}(k,\chi_0)}\prod_{\mathfrak{p}|\mathfrak{q}}\left(1-\frac{1}{\mathcal{N}(\mathfrak{p})}\right)+x^{1+\frac{\Theta}{k}-\epsilon}, \]
where $\Theta$ denotes the supremum of real part of the zeros of the Hecke $L$-function mod $\mathfrak{q}$. For $\sigma>2$, we consider the integral
\begin{align*}
    \int_1^{\infty}\mathfrak{A}_{\mathbb{K}}(x)x^{-s-1}dx
    %&=\int_1^{\infty}\left(-\mathcal{S}_{\mathbb{K}}(x;\mathfrak{q},\mathfrak{a})+\frac{x^2\lambda_{\mathbb{K}}{R}_{k}(2,\chi_0)}{2|\mathcal{O}_\mathbb{K}/\mathfrak{q}|\mathcal{L}(2,\chi_0)}\prod_{\mathfrak{p}|\mathfrak{q}}\left(1-\frac{1}{\mathcal{N}(\mathfrak{p})}\right)+x^{1+\frac{\Theta}{k}-\epsilon} \right)x^{-s-1}ds\\
    &=\frac{-1}{|(\mathcal{O}_\mathbb{K}/\mathfrak{q})^*|}\sum_{\chi\pmod{\mathfrak{q}}}\bar{\chi}(\mathfrak{a})\sum_{\substack{\mathfrak{I}\subset\mathcal{O}_{\mathbb{K}}}}\chi(\mathfrak{I})\mathcal{F}_{k}(\mathfrak{I})\int_{\mathcal{N}\mathfrak{I}}^{\infty}\frac{dx}{x^{s+1}}+\frac{\lambda_{\mathbb{K}}{R}_{k}(2,\chi_0)}{2|(\mathcal{O}_\mathbb{K}/\mathfrak{q})^*|\mathcal{L}(k,\chi_0)}\\&\times\prod_{\mathfrak{p}|\mathfrak{q}}\left(1-\frac{1}{\mathcal{N}(\mathfrak{p})}\right)\int_1^{\infty}x^{1-s}dx+\frac{1}{s-1-\frac{\Theta}{k}+\epsilon}\\
    %&=\frac{-1}{s|\mathcal{O}_\mathbb{K}/\mathfrak{q}|}\sum_{\chi\pmod{\mathfrak{q}}}\bar{\chi}(\mathfrak{a})\sum_{\substack{\mathfrak{I}\subset\mathcal{O}_{\mathbb{K}}}}\frac{\chi(\mathfrak{I})\mathcal{F}_{k}(\mathfrak{I})}{(\mathcal{N}\mathfrak{I})^s}+\frac{\lambda_{\mathbb{K}}{R}_{k}(2,\chi_0)}{2(s-2)|\mathcal{O}_\mathbb{K}/\mathfrak{q}|\mathcal{L}(k,\chi_0)}\prod_{\mathfrak{p}|\mathfrak{q}}\left(1-\frac{1}{\mathcal{N}(\mathfrak{p})}\right)+\frac{1}{s-1-\frac{\Theta}{k}+\epsilon}\\
    &=\frac{-1}{s|(\mathcal{O}_\mathbb{K}/\mathfrak{q})^*|}\sum_{\substack{\chi\pmod{\mathfrak{q}}\\\chi\ne\chi_0}}\frac{\bar{\chi}(\mathfrak{a})\mathcal{L}(s-1,\chi)}{\mathcal{L}(ks-k,\chi^k)}R_k(s,\chi)-\frac{\mathcal{L}(s-1,\chi_0)R_k(s,\chi_0)}{s|(\mathcal{O}_\mathbb{K}/\mathfrak{q})^*|\mathcal{L}(ks-k,\chi_0)}\\&+\frac{\lambda_{\mathbb{K}}{R}_{k}(2,\chi_0)}{2(s-2)|(\mathcal{O}_\mathbb{K}/\mathfrak{q})^*|\mathcal{L}(k,\chi_0)}\prod_{\mathfrak{p}|\mathfrak{q}}\left(1-\frac{1}{\mathcal{N}(\mathfrak{p})}\right)+\frac{1}{s-1-\frac{\Theta}{k}+\epsilon}.\numberthis\label{N4}
    \end{align*}
    Note that the second term on the right hand side of \eqref{N4} has a simple pole at $s=2$ and it will be canceled by the pole of third term. Therefore, the right hand side of \eqref{N4} is analytic in the half plane $\Re(s)>1+\frac{1}{k}$ and assuming Haselgrove's condition for Hecke $L$-function modulo $\mathfrak{q}$ it can be analytically continued to the real segment $(1+\frac{\Theta}{k}-\epsilon, 1+\frac{1}{k}]$ with a pole at $s=1+\frac{\Theta}{k}-\epsilon$. This yields 
$\int_1^{\infty}\mathfrak{A}_{\mathbb{K}}(x)x^{-\sigma-1}dx<\infty$ for $\sigma>1+\frac{\Theta}{k}-\epsilon$. Suppose that 
\[\mathcal{S}_{k,\mathbb{K}}(x;\mathfrak{q},\mathfrak{a})-\frac{x^2\lambda_{\mathbb{K}}{R}_{k}(2,\chi_0)}{2|(\mathcal{O}_\mathbb{K}/\mathfrak{q})^*|\mathcal{L}(k,\chi_0)}\prod_{\mathfrak{p}|\mathfrak{q}}\left(1-\frac{1}{\mathcal{N}(\mathfrak{p})}\right)<x^{1+\frac{\Theta}{k}-\epsilon}\ \text{for all}\ x>x_0(\epsilon) .\numberthis\label{N5}\]
We apply Proposition \ref{omega} for $\mathfrak{A}_{\mathbb{K}}(x)$ and obtain that the integral $\int_1^{\infty}\mathfrak{A}_{\mathbb{K}}(x)x^{-s-1}dx$ is analytic in the half plane $\Re(s)>1+\frac{\Theta}{k}-\epsilon$. In view of the definition of $\Theta$, the right hand side of \eqref{N4} has pole at $\Re(s)=1+\frac{\Theta}{k}$ arising from the zeros of the Hecke $L$-function $\mathcal{L}(ks-k,\chi^k)$. This leads to a contradiction, and we can deduce that the assumption in \eqref{N5} is false. Therefore,
\[\mathcal{S}_{k,\mathbb{K}}(x;\mathfrak{q},\mathfrak{a})-\frac{x^2\lambda_{\mathbb{K}}{R}_{k}(2,\chi_0)}{2|(\mathcal{O}_\mathbb{K}/\mathfrak{q})^*|\mathcal{L}(k,\chi_0)}\prod_{\mathfrak{p}|\mathfrak{q}}\left(1-\frac{1}{\mathcal{N}(\mathfrak{p})}\right)=\Omega_{+}\left(x^{1+\frac{\Theta}{k}-\epsilon}\right). \]
To obtain the corresponding estimate for $\Omega_-$, we proceed in a similar way with the following identity
\[\mathfrak{A}_{\mathbb{K}}(x)=\mathcal{S}_{\mathbb{K}}(x;\mathfrak{q},\mathfrak{a})-\frac{x^2\lambda_{\mathbb{K}}{R}_{k}(2,\chi_0)}{2|(\mathcal{O}_\mathbb{K}/\mathfrak{q})^*|\mathcal{L}(k,\chi_0)}\prod_{\mathfrak{p}|\mathfrak{q}}\left(1-\frac{1}{\mathcal{N}(\mathfrak{p})}\right)-x^{1+\frac{\Theta}{k}-\epsilon}, \]
and get
\[\mathcal{S}_{k,\mathbb{K}}(x;\mathfrak{q},\mathfrak{a})-\frac{x^2\lambda_{\mathbb{K}}{R}_{k}(2,\chi_0)}{2|(\mathcal{O}_\mathbb{K}/\mathfrak{q})^*|\mathcal{L}(k,\chi_0)}\prod_{\mathfrak{p}|\mathfrak{q}}\left(1-\frac{1}{\mathcal{N}(\mathfrak{p})}\right)=\Omega_{-}\left(x^{1+\frac{\Theta}{k}-\epsilon}\right). \]
This completes the proof of Theorem \ref{thm5}.

\subsection{Proof of Theorem \ref{thm6}}
In order to prove Theorem \ref{thm6}, we use Knapowski and Turan \cite{Knapowski} approach. We denote $\mathcal{A}(x)=\mathcal{S}_{k,\mathbb{K}}(x;\mathfrak{q},\mathfrak{a}_1)-\mathcal{S}_{k,\mathbb{K}}(x;\mathfrak{q},\mathfrak{a}_2)\pm cx^{1+\frac{1}{2k}-\epsilon}$ and begin with the integral 
\[g(s)=\int_1^{\infty}\mathcal{A}(x)x^{-s-1}dx,\ \text{for}\ \epsilon>0, \]
where
\[\mathcal{S}_{k,\mathbb{K}}(x;\mathfrak{q},\mathfrak{a}_1)-\mathcal{S}_{k,\mathbb{K}}(x;\mathfrak{q},\mathfrak{a}_2)=\frac{1}{|(\mathcal{O}_{\mathbb{K}}/ \mathfrak{q})^*|}\sum_{\chi\pmod{\mathfrak{q}}}(\bar{\chi}(\mathfrak{a}_1)-\bar{\chi}(\mathfrak{a}_2))\sum_{\substack{\mathfrak{I}\subset\mathcal{O}_{\mathbb{K}}\\ \mathcal{N}(\mathfrak{I})\leq x}}\chi(\mathfrak{I})\mathcal{F}_{k}(\mathfrak{I}). \]
Therefore,
\begin{align*}
    g(s)&=\frac{1}{|(\mathcal{O}_{\mathbb{K}}/ \mathfrak{q})^*|}\sum_{\chi\pmod{\mathfrak{q}}}(\bar{\chi}(\mathfrak{a}_1)-\bar{\chi}(\mathfrak{a}_2))\int_1^{\infty}\frac{1}{x^{s+1}}\sum_{\substack{\mathfrak{I}\subset\mathcal{O}_{\mathbb{K}}\\ \mathcal{N}(\mathfrak{I})\leq x}}\chi(\mathfrak{I})\mathcal{F}_{k}(\mathfrak{I})dx\pm \frac{c}{s-1-\frac{1}{2k}+\epsilon}.\numberthis\label{N7}
\end{align*}
Note that the integral on the right hand side of the above identity represents the Dirichlet series $\sum_{\substack{\mathfrak{I}\subset\mathcal{O}_{\mathbb{K}}}}\frac{\chi(\mathfrak{I})\mathcal{F}_{k}(\mathfrak{I})}{(\mathcal{N}\mathfrak{I})^s}$ as a Mellin transform. Thus
\begin{align*}
  \int_1^{\infty}\frac{1}{x^{s+1}}\sum_{\substack{\mathfrak{I}\subset\mathcal{O}_{\mathbb{K}}\\ \mathcal{N}(\mathfrak{I})\leq x}}\chi(\mathfrak{I})\mathcal{F}_{k}(\mathfrak{I})dx&=\frac{1}{s}\sum_{\substack{\mathfrak{I}\subset\mathcal{O}_{\mathbb{K}}}}\frac{\chi(\mathfrak{I})\mathcal{F}_{k}(\mathfrak{I})}{(\mathcal{N}\mathfrak{I})^s}
  =\frac{\mathcal{L}(s-1,\chi)}{s\mathcal{L}(ks-k,\chi^k)}R_k(s,\chi),\ \text{for}\ \sigma>2.
\end{align*}
In the last step, we used equation \eqref{N6}.
Inserting the above estimate into \eqref{N7} yields
\begin{align*}
   {g}(s)&=\frac{1}{|(\mathcal{O}_{\mathbb{K}}/ \mathfrak{q})^*|}\sum_{\chi\pmod{\mathfrak{q}}}(\bar{\chi}(\mathfrak{a}_1)-\bar{\chi}(\mathfrak{a}_2))\frac{\mathcal{L}(s-1,\chi)}{s\mathcal{L}(ks-k,\chi^k)}R_k(s,\chi) \pm \frac{c}{s-1-\frac{1}{2k}+\epsilon}\\
   &=\frac{1}{|(\mathcal{O}_{\mathbb{K}}/ \mathfrak{q})^*|}\sum_{\substack{\chi\pmod{\mathfrak{q}}\\\chi\ne\chi_0}}(\bar{\chi}(\mathfrak{a}_1)-\bar{\chi}(\mathfrak{a}_2))\frac{\mathcal{L}(s-1,\chi)}{s\mathcal{L}(ks-k,\chi^k)}R_k(s,\chi) \pm \frac{c}{s-1-\frac{1}{2k}+\epsilon}.
\end{align*}
In here, we utilized the fact that $\bar{\chi}(\mathfrak{a}_1)=\bar{\chi}(\mathfrak{a}_1)=1$ for $\gcd(\mathfrak{a}_1\mathfrak{a}_2,\mathfrak{q})=1$. 
Note that the product term $R_{k,\chi}(s)$ is absolutely convergent for $\Re(\sigma)>1$. It is well known that the Hecke $L$-function is analytic in the half-plane $\Re(s)$ for $\chi\ne\chi_0$, and the denominator $\mathcal{L}(ks-k,\chi^k)$ has no zero (see Davenport \cite{Davenport}, p. 84-85) in the region $\Re(s)\geq 1+\frac{1}{k}$. Thus, $g(s)$ is analytic in the half plane $\Re(s)\geq 1+\frac{1}{k}$.

Using the Haselgrove's condition, the Hecke $L$-function $\mathcal{L}(ks-k,\chi^k)\ne 0$ on the real segment $(1+\frac{1}{2k}-\epsilon, 1+\frac{1}{k}]$, it implies that $1/\mathcal{L}(ks-k,\chi^k)$ admits an analytic continuation on the real segment $1+\frac{1}{2k}-\epsilon<\sigma<1+\frac{1}{k}$.
 Therefore, $g(s)$ is regular on the real segment $1+\frac{1}{2k}-\epsilon<\sigma<1+\frac{1}{k}$.

 Suppose there is a positive constant $x_0$ such that $\mathcal{A}(x)$ does not change the sign for $x>x_0$. Then by Proposition \ref{Landau}, $g(s)$ is analytic in the half-plane $\Re(s)>1+\frac{1}{2k}-\epsilon$. Therefore, each zero of the denominator $\mathcal{L}(ks-k,\chi^k)$ is canceled by the zeros of numerator $\mathcal{L}(s-1,\chi)$. However, it is well known that almost all zeros of $\mathcal{L}(ks-k,\chi^k)$ are close to $\Re(s)=1+\frac{1}{2k}$. Similarly, all non-trivial zeros of $\mathcal{L}(s-1,\chi)$ are close to $\sigma=\frac{3}{2}$. Hence, the zeros of $\mathcal{L}(ks-k,\chi^k)$ that are canceled by the zeros of $\mathcal{L}(s-1,\chi)$ are negligible. Any zero of the denominator $\mathcal{L}(ks-k,\chi^k)$ which is not canceled by the zeros of numerator $\mathcal{L}(s-1,\chi)$ is a pole of $g(s)$, and this contradicts the assumption on the existence of a positive constant $x_0$ such that $\mathcal{A}(x)$ does not change the sign for $x>x_0$, through Proposition \ref{Landau}. This completes the proof of Theorem \ref{thm6}.

 \section{Function fields}
 \subsection{Proof of Theorem \ref{thm7}}
 We begin with the sum
 \begin{align*}
    \mathfrak{S}_{k}(N; \mathfrak{m},\mathfrak{g})&=\sum_{\substack{f\in\mathbb{F}_q[X]\\\deg(f)\leq N\\f\equiv\mathfrak{g}\pmod{\mathfrak{m}}}}\mathfrak{F}_{k}(f)
     %=\sum_{\substack{f\in\mathbb{F}_q[X]\\\deg(f)\leq N}}\mathfrak{F}_{k}(f)\frac{1}{\Phi(\mathfrak{m})}\sum_{\chi}\bar{\chi}(\mathfrak{g})\chi(f)\\
     =\frac{1}{\Phi(\mathfrak{m})}\sum_{\chi\pmod{\mathfrak{m}}}\bar{\chi}(\mathfrak{g})\sum_{\substack{f\in\mathbb{F}_q[X]\\\deg(f)\leq N}}\chi(f)\mathfrak{F}_{k}(f).\numberthis\label{N8}
 \end{align*}
In the last step, we used the orthogonality relation of Dirichlet characters modulo $\mathfrak{m}$ over the polynomial ring $\mathbb{F}_q[x]$. Since the arithmetic function $\chi(f)\mathfrak{F}_{k}(f)$ is multiplicative, its Dirichlet series has an Euler product representation and is given by
 \begin{align*}
     F(s)&=\sum_{\substack{f\in\mathbb{F}_q[X]\\ f\  \text{monic}}}\frac{\chi(f)\mathfrak{F}_{k}(f)}{|f|^s}
     %=\prod_{\substack{P\in\mathbb{F}_q[X]\\ \text{monic, irreducible}}}\left(1+\frac{\chi(P)\mathfrak{F}_{k}(P)}{|P|^s}+\frac{\chi(P^2)\mathfrak{F}_{k}(P^2)}{|P|^{2s}}+\cdots+\frac{\chi(P^k)\mathfrak{F}_{k}(P)}{|P|^{ks}}+\cdots \right)\\
    %&=\prod_{\substack{P\in\mathbb{F}_q[X]\\ \text{monic, irreducible}}}\left(1+\frac{\chi(P)}{|P|^{s-1}}+\frac{\chi(P^2)}{|P|^{2s-2}}+\cdots+\frac{\chi(P^{k-1})}{|P|^{(k-1)(s-1)}}+\frac{\chi(P^k)}{|P|^{ks-\frac{1}{k}}}+\frac{\chi(P^{k+1})}{|P|^{(k+1)s-\frac{1}{k+1}}}+\cdots \right)\\
    =\frac{L_q(s-1,\chi)}{L_q(ks-k,\chi^k)}\mathcal{M}_k(s,\chi),\numberthis\label{37}
 \end{align*}
 where $\mathcal{M}_k(s,\chi)=\prod_{\substack{P\in\mathbb{F}_q[X]\\ \text{monic, irreducible}}}(1+\mathscr{A}_{P, \chi}(s))$ and
\[\mathscr{A}_{P, \chi}(s)=\frac{\frac{1}{|P|^{ks-\frac{1}{k}}}\sum_{m=0}^{\infty}\frac{\chi(P^{k+m})}{|P|^{ms+\frac{1}{k}-\frac{1}{k+m}}}}{\left(1+\frac{\chi(P)}{|P|^{s-1}}+\frac{\chi(P^2)}{|P|^{2s-2}}+\cdots+\frac{\chi(P^{k-1})}{|P|^{(k-1)(s-1)}} \right)} .\]
Note that $\mathcal{M}_k(s,\chi)$ is absolutely convergent and defines an analytic function on $\Re(s)>1$. Also, the Dirichlet $L$-function $L_q(s,\chi)$ modulo $\mathfrak{m}$ for polynomial ring has meromorphic continuation to the complex plane, with a simple pole at $s=1$ only for the principal character $\chi=\chi_0$. It is well known that the Riemann hypothesis is true for function fields, thus the denominator $L_q(ks-k,\chi^k)$ has no zero in the half-plane $\Re(s)> 1+\frac{1}{2k}$. Therefore, the Dirichlet series $F(s)$ has an analytic continuation to $\Re(s)>1+\frac{1}{2k}$, except for a simple pole at $s=2$ associated with the principal character $\chi=\chi_0$. Employing Lemma\ref{prop7} with $a(f)=\chi(f)\mathfrak{F}_{k}(f)$ and $\alpha=2+\frac{1}{ N}$ gives
\begin{align*}
    \sum_{\substack{f\in\mathbb{F}_q[T]\\\deg(f)\leq N}}\chi(f)\mathfrak{F}_{k}(f)&=\frac{1}{2\pi i}\int_{\alpha-iT}^{\alpha+iT}F(s)\frac{q^{Ns}}{s}ds+\BigO{\frac{q^{\alpha N}}{T}\sum_{f}\frac{\mathfrak{F}_{k}(f)}{|f|^{\alpha}|\log(q^N/|f|)|}}\\
    &=\frac{1}{2\pi i}\int_{\alpha-iT}^{\alpha+iT}F(s)\frac{q^{Ns}}{s}ds+\BigO{\frac{q^{2N}\log N}{T}}.\numberthis\label{N21}
\end{align*}
To estimate the line integral in the above identity, we deform the line of integration into a rectangular contour with vertices $1+1/2k+\epsilon\pm iT$ and $\alpha\pm iT$.

Case I: For $\chi=\chi_0$. By Cauchy's residue theorem, we have
\[\frac{1}{2\pi i}\int_{\alpha-iT}^{\alpha+iT}F(s)\frac{q^{Ns}}{s}ds=\frac{q^{2N}\mathcal{M}_k(s,\chi_0)}{2\log qL_q(k,\chi_0)}\prod_{P|\mathfrak{m}}\left(1-\frac{1}{|P|} \right)+\sum_{j=1}^3I_j, \numberthis\label{N20}\]
where $I_1, I_2$, and $I_3$ are integrals along the lines $[\alpha-iT, 1+1/2k+\epsilon-iT]$, $[1+1/2k+\epsilon-iT, 1+1/2k+\epsilon+iT]$, and $[1+1/2k+\epsilon+iT, \alpha+iT]$, respectively. We next estimate the integrals on the right-hand side of the above identity. Using \eqref{zetaq}, we obtain
\[I_1, I_3\ll_{\mathfrak{m},k} \frac{1}{T}\left(\int_{1+\frac{1}{2k}+\epsilon}^2q^{N\sigma}T^{\epsilon}d\sigma +\int_{2}^{\alpha}q^{N\sigma}d\sigma\right)\ll_{\mathfrak{m}, k} \frac{q^{2N}}{T^{1-\epsilon}}\ \text{and}\ I_2\ll_{\mathfrak{m}, k} q^{N(1+\frac{1}{2k})+\epsilon}\int_{0}^{T}\frac{1}{t^{1-\epsilon}}dt\ll_{\mathfrak{m}, k} q^{N(1+\frac{1}{2k})+\epsilon} T^{\epsilon}. \]
%and
%\[I_2\ll q^{N(1+\frac{1}{2k})+\epsilon}\int_{0}^{T}\frac{1}{t^{1-\epsilon}}dt\ll q^{N(1+\frac{1}{2k})+\epsilon} T^{\epsilon}.  \]
Case II: For $\chi\ne\chi_0$, using Cauchy's residue theorem, we have
\[\frac{1}{2\pi i}\int_{\alpha-iT}^{\alpha+iT}F(s)\frac{q^{Ns}}{s}ds=\sum_{j=1}^3I_j, \numberthis\label{N22}\]
where $I_j$ are same as in Case I. The above integrals are estimated in a similar way to those for principal characters. Therefore, by \eqref{L_k}, we have
\[I_1, I_3\ll_{\mathfrak{m}, k} \frac{q^{2N}}{T^{1-\epsilon}}\ \text{and}\ I_2\ll_{\mathfrak{m}, k} q^{N(1+\frac{1}{2k})+\epsilon}T^{\epsilon}.  \]
Inserting the above estimates from \eqref{N20} and \eqref{N22} into \eqref{N21} and setting $T=q^{N-\frac{N}{2k}-\epsilon}$, we obtain the required result.

\subsection{Proof of Theorem \ref{thm9}}
The key idea to obtain the omega result is the Proposition \ref{omega}. We denote
\[\mathscr{A}(N)=-\mathfrak{S}_{k}(N; \mathfrak{m},\mathfrak{g})+\frac{q^{2N}\mathcal{M}_k(s,\chi_0)}{2\Phi(\mathfrak{m})\log qL_q(k,\chi_0)}\prod_{P|\mathfrak{m}}\left(1-\frac{1}{|P|} \right)+q^{N(1+\frac{1}{2k}-\epsilon)}. \]
We begin with the integral
\begin{align*}
    \int_1^{\infty}\frac{\mathscr{A}(N)}{q^{N(s+1)}}dq^N
    %&=\int_1^{\infty}\left(-\mathfrak{S}_{k}(N; \mathfrak{m},\mathfrak{g})+\frac{q^{2N}}{2\Phi(\mathfrak{m})\log qL_q(k,\chi_0)}\prod_{P|\mathfrak{m}}\left(1-\frac{1}{|P|} \right)+q^{N(1+\frac{1}{2k}-\epsilon)} \right)q^{-N(s+1)}dq^N\\
    &=\int_1^{\infty}\left(-\frac{1}{\Phi(\mathfrak{m})}\sum_{\chi\pmod{\mathfrak{m}}}\bar{\chi}(\mathfrak{g})\sum_{\substack{f\in\mathbb{F}_q[X]\\\deg(f)\leq N}}\chi(f)\mathfrak{F}_{k}(f) \right)q^{-N(s+1)}dq^N\\&+ \frac{\mathcal{M}_k(s,\chi_0)}{2(s-2)\Phi(\mathfrak{m})\log qL_q(k,\chi_0)}\prod_{P|\mathfrak{m}}\left(1-\frac{1}{|P|} \right)+\frac{1}{s-1-\frac{1}{2k}+\epsilon}\\
    %&=\frac{-1}{s\Phi(\mathfrak{m})}\sum_{\chi}\bar{\chi}(\mathfrak{g})\sum_{\substack{f\in\mathbb{F}_q[X]\\\deg(f)\leq N}}\frac{\chi(f)\mathfrak{F}_{k}(f)}{|f|^s}+ \frac{\mathcal{M}_k(s,\chi)}{2(s-2)\Phi(\mathfrak{m})\log qL_q(k,\chi_0)}\prod_{P|\mathfrak{m}}\left(1-\frac{1}{|P|} \right)\\&+\frac{1}{s-1-\frac{1}{2k}+\epsilon}\\
    &=\frac{-1}{s\Phi(\mathfrak{m})}\sum_{\substack{\chi\pmod{\mathfrak{m}}\\\chi\ne\chi_0}}\frac{\bar{\chi}(\mathfrak{g})L_q(s-1,\chi)}{L_q(ks-k,\chi^k)}\mathcal{M}_k(s,\chi)-\frac{L_q(s-1,\chi_0)}{s\Phi(\mathfrak{m})L_q(ks-k,\chi_0)}\mathcal{M}_k(s,\chi_0)\\&+\frac{\mathcal{M}_k(s,\chi_0)}{2(s-2)\Phi(\mathfrak{m})\log qL_q(k,\chi_0)}\prod_{P|\mathfrak{m}}\left(1-\frac{1}{|P|} \right)+\frac{1}{s-1-\frac{1}{2k}+\epsilon}.\numberthis\label{F1}
\end{align*}
We observe that the second term on the right-hand side of the above identity has a simple pole at $s=2$, which is canceled by the pole of the third term at $s=2$. Therefore, the right hand side of the integral in \eqref{F1} is analytic in the half-plane $\Re(s)> 1+\frac{1}{2k}$. It is well known that the Riemann hypothesis holds in function fields. Assuming $L_q(\frac{1}{2},\chi)\ne 0$, the integral in \eqref{F1} can be analytically continued to the real segment $(1+\frac{1}{2k}-\epsilon,1+\frac{1}{2k}]$, with a simple pole at $s=1+\frac{1}{2k}-\epsilon$. This, in turn, implies that $\int_1^{\infty}\mathscr{A}(N)q^{-N(\sigma+1)}dq^N<\infty$ for $\sigma>1+\frac{1}{2k}-\epsilon$. Suppose that
\[\mathfrak{S}_{k}(N; \mathfrak{m},\mathfrak{g})-\frac{q^{2N}\mathcal{M}_k(s,\chi_0)}{2\Phi(\mathfrak{m})\log qL_q(k,\chi_0)}\prod_{P|\mathfrak{m}}\left(1-\frac{1}{|P|} \right)<q^{N(1+\frac{1}{2k}-\epsilon)}\ \text{for all}\ N>N_0. \numberthis\label{F2}\]
We apply Proposition \ref{omega} for $\mathscr{A}(N)$ which in turn yields that the integral $\int_1^{\infty}\mathscr{A}(N)q^{-N(s+1)}dq^N$ is analytic in the half plane $\Re(s)>1+\frac{1}{2k}-\epsilon$. Since the integral has pole at $s=1+\frac{1}{2k}$, which arises from the zeros of $L_q(ks-k,\chi^k)$. This is a contradiction and one may deduce that the assertion in \eqref{F2} is not true. Therefore,
\[\mathfrak{S}_{k}(N; \mathfrak{m},\mathfrak{g})-\frac{q^{2N}\mathcal{M}_k(s,\chi_0)}{2\Phi(\mathfrak{m})\log qL_q(k,\chi_0)}\prod_{P|\mathfrak{m}}\left(1-\frac{1}{|P|} \right)=\Omega_+(q^{N(1+\frac{1}{2k}-\epsilon)}). \]
Similarly, we obtain the corresponding $\Omega_-$ bound
\[\mathfrak{S}_{k}(N; \mathfrak{m},\mathfrak{g})-\frac{q^{2N}\mathcal{M}_k(s,\chi_0)}{2\Phi(\mathfrak{m})\log qL_q(k,\chi_0)}\prod_{P|\mathfrak{m}}\left(1-\frac{1}{|P|} \right)=\Omega_-(q^{N(1+\frac{1}{2k}-\epsilon)}). \]
This completes the proof of Theorem \ref{thm9}.

\subsection{Proof of Theorem \ref{thm8}}
We apply the techniques developed by Knapowski and Turan \cite{Knapowski} to investigate the sign changes of $\mathfrak{S}_{k}(N; \mathfrak{m},\mathfrak{g}_1)-\mathfrak{S}_{k}(N; \mathfrak{m},\mathfrak{g}_2)$, and, using \eqref{N8}, the difference can be expressed in the following form:
\[ \mathfrak{S}_{k}(N; \mathfrak{m},\mathfrak{g}_1)-\mathfrak{S}_{k}(N; \mathfrak{m},\mathfrak{g}_2)=\frac{1}{\Phi(\mathfrak{m})}\sum_{\chi\pmod{\mathfrak{m}}}(\bar{\chi}(\mathfrak{g}_1)-\bar{\chi}(\mathfrak{g}_2))\sum_{\substack{f\in\mathbb{F}_q[X]\\\deg(f)\leq N}}\chi(f)\mathfrak{F}_{k}(f). \]
We begin with the integral
\begin{align*}
   \mathscr{I}(s)&= \int_1^{\infty}(\mathfrak{S}_{k}(N; \mathfrak{m},\mathfrak{g}_1)-\mathfrak{S}_{k}(N; \mathfrak{m},\mathfrak{g}_2)\pm c(q^N)^{1+\frac{1}{2k}-\epsilon})(q^N)^{-s-1}dq^N\\
   &=\frac{1}{\Phi(\mathfrak{m})}\sum_{\chi\pmod{m}}(\bar{\chi}(\mathfrak{g}_1)-\bar{\chi}(\mathfrak{g}_2))\int_1^{\infty}\frac{1}{q^{N(s+1)}}{\sum_{\substack{f\in\mathbb{F}_q[X]\\\deg(f)\leq N}}\chi(f)\mathfrak{F}_k(f)}dq^N\pm \frac{c}{s-1-\frac{1}{2k}+\epsilon}.\numberthis\label{N9}
    %&=\frac{1}{\phi(q)}\sum_{\chi}(\bar{\chi}(a_1)-\bar{\chi}(a_2))\sum_{n=1}^{\infty}\chi(n)I_k(n)\int_{n}^{\infty}Q^{-s-1}dQ
\end{align*}
Since the integral on the right-hand side of \eqref{N9} represents the Dirichlet series $\sum_{\substack{f\in\mathbb{F}_q[X]\\ \text{monic}}}\frac{\chi(f)\mathfrak{F}_{k}(f)}{|f|^s}$ as a Mellin transform. Thus
\begin{align*}
 \int_1^{\infty}\frac{1}{q^{N(s+1)}}{\sum_{\substack{f\in\mathbb{F}_q[X]\\\deg(f)\leq N}}\chi(f)\mathfrak{F}_{k}(f)}dq^N &= \frac{1}{s}\sum_{\substack{f\in\mathbb{F}_q[X]\\ \text{monic}}}\frac{\chi(f)\mathfrak{F}_{k}(f)}{|f|^s}
  =\frac{L_q(s-1,\chi)}{L_q(ks-k,\chi^k)}\mathcal{M}_k(s,\chi).
 \end{align*} 
In the last step, we used \eqref{37}. Plugging the above identity into \eqref{N9} yields
\begin{align*}
\mathscr{I}(s)&=\frac{1}{s\Phi(\mathfrak{m})}\sum_{\chi\pmod{\mathfrak{m}}}(\bar{\chi}(\mathfrak{g}_1)-\bar{\chi}(\mathfrak{g}_2))\frac{L_q(s-1,\chi)}{L_q(ks-k,\chi^k)}\mathcal{M}_k(s,\chi)\pm \frac{c}{s-1-\frac{1}{2k}+\epsilon} \\
&=\frac{1}{s\Phi(\mathfrak{m})}\sum_{\substack{\chi\pmod{m}\\\chi\ne\chi_0}}(\bar{\chi}(\mathfrak{g}_1)-\bar{\chi}(\mathfrak{g}_2))\frac{L_q(s-1,\chi)}{L_q(ks-k,\chi^k)}\mathcal{M}_k(s,\chi)\pm \frac{c}{s-1-\frac{1}{2k}+\epsilon}.
\end{align*}
We apply Proposition \ref{Landau} with $\mathbb{A}(N)=\mathfrak{S}_{k}(N; \mathfrak{m},\mathfrak{g}_1)-\mathfrak{S}_{k}(N; \mathfrak{m},\mathfrak{g}_2)\pm cq^{N(1+\frac{1}{2k}-\epsilon)}$. Thus
\begin{align*}
    \mathscr{I}(s)=\int_1^{\infty}\mathbb{A}(N)q^{-N(s+1)}dq^N
    =\frac{1}{\Phi(\mathfrak{m})}\sum_{\substack{\chi\pmod{\mathfrak{m}}\\\chi\ne\chi_0}}(\bar{\chi}(g_1)-\bar{\chi}(g_2))\frac{L_q(s-1,\chi)\mathcal{M}_k(s,\chi)}{L_q(ks-k,\chi^k)}\pm \frac{c}{s-1-\frac{1}{2k}+\epsilon}.
\end{align*}
Following the argument in the proof of Theorem \ref{thm7}, we note that the integral $\mathscr{I}(s)$ in the above identity is analytic in the half-plane $\Re(s)> 1+\frac{1}{2k}$.
Assuming that $L_q(1/2,\chi)\ne 0$, the function
 $1/L_q(ks-k,\chi^k)$ defines an holomorphic function on the real segment $1+\frac{1}{2k}-\epsilon<\sigma\leq 1+\frac{1}{2k}$, consequently $\mathscr{I}(s)$ has analytic continuation on the real segment $1+\frac{1}{2k}-\epsilon<\sigma\leq 1+\frac{1}{2k}$.

%Note that the Euler product $\mathcal{M}_k(s,\chi)$ is absolutely convergent for $\Re(s)>1$, it is well known that $L_q(s-1,\chi)$ is analytic in the half-plane $\Re(s)>2$ and can be analytically continued to the complex plane except for a simple pole at $s=2$ for the principal character $\chi=\chi_0$. Since the Riemann hypothesis is true for function fields, thus the denominator $L_q(ks-k,\chi^k)$ is non-zero in the half-plane $\Re(s)> 1+\frac{1}{2k}$. Therefore, $\mathscr{I}(s)$ is analytic in the half plane $\Re(s)> 1+\frac{1}{2k}$.

 Suppose there exists a positive constant $N_0$ such that for $N>N_0$, $\mathbb{A}(N)$ preserves its sign. By Proposition \ref{Landau}, $\mathscr{I}(s)$ is analytic in the half-plane $\Re(s)>1+\frac{1}{2k}-\epsilon$. Therefore, each zero of $L_q(ks-k,\chi^k)$ is canceled by the zeros of $L_q(s-1,\chi)$. However, it is well known that all zeros of $L_q(ks-k,\chi^k)$ lie on the line $\Re(s)=1+\frac{1}{2k}$, while all zeros of $L_q(s-1,\chi)$ are on the line $\sigma=\frac{3}{2}$. Hence, the zeros of the denominator $L_q(ks-k,\chi^k)$ cannot be canceled by the zeros of the numerator $L_q(s-1,\chi)$. Any such zero of the denominator $L_q(ks-k,\chi^k)$ is a pole of $\mathscr{I}(s)$, which, by Proposition \ref{Landau}, contradicts the assumption that $\mathbb{A}(N)$ does not change the sign for $N>N_0$. This completes the proof of Theorem \ref{thm8}.

\bibliographystyle{plain}
\bibliography{reference}
\end{document}